\documentclass[12pt]{article} %amsart, article, ...
\usepackage{geometry}
\usepackage{graphicx}
\usepackage{amsmath,amssymb,amsthm}
\usepackage[small,nohug,heads=vee]{diagrams}
\diagramstyle[labelstyle=\scriptstyle]

\newtheorem{theorem}{Theorem}[section]
\newtheorem{lemma}[theorem]{Lemma}

\newtheorem{corollary}[theorem]{Corollary}
\newtheorem{proposition}[theorem]{Proposition}
\theoremstyle{definition}
\newtheorem{example}[theorem]{Example}
\newtheorem{definition}[theorem]{Definition}
\numberwithin{equation}{section}

\def \tri {\ensuremath{\triangle}}
\def \bb {\textbf}
\def \ii {\textit}

\def\({\left(}
\def\){\right)}

\newcommand{\cT}{\mathcal{T}}

\newcommand{\cA}{\mathcal{A}}

\newcommand{\nn}{\mathbb{N}}
\newcommand{\zz}{\mathbb{Z}}

\newcommand{\rr}{\mathbb{R}}

\def\bl{\boldsymbol{l}}
\def\bt{\boldsymbol{t}}
\def\bk{\boldsymbol{k}}
\def\bbf{\boldsymbol{f}}

\newcommand{\la}{\lambda }
\newcommand{\de}{\delta }

\newcommand{\si}{\sigma }

\newcommand{\om}{\omega }
\newcommand{\Om}{\Omega }

\newcommand{\arr}[2]{\xrightarrow[#2]{#1}}

\newcommand{\boxif}[1]{
\ifx\\#1\\
\else%
\fbox{\scriptsize #1}%
\fi
}

\newcommand{\freq}{\op{freq}}

\def\<{\left\langle }
\def\>{\right\rangle}

\def\<-{\leftarrow }
\def\->{\rightarrow}
\def\<={\Leftarrow }
\def\=>{\Rightarrow}
\def\op{\operatorname}

\def\pig{\tilde\pi_g}
\def\lag{\tilde\lambda}
\def\pigs{\tilde\pi_{g*}}
\def\lags{\tilde\lambda_{*}}

\def\ogx{{\Omega_{g,x}}}

\def\AA{\ensuremath{\mathcal{A}}}

\def\GG{\ensuremath{\mathcal{G}}}

\title{On  Mealy-Moore  coding  and\\  images  of  Markov  measures}
\author{Rostislav Grigorchuk\thanks{The first author graciously acknowledges
support from the Simons Foundation through Collaboration Grant \#527814.}, 
Roman Kogan and Yaroslav Vorobets}
\date{}

\begin{document}

\maketitle

\begin{abstract}
We study the images of the Markov measures under transformations generated
by the Mealy automata.  We find conditions under which the image measure is 
absolutely continuous or singular relative to the Markov measure.  Also, we 
determine statistical properties of the image of a generic sequence.

11 references, 7 figures.  UDC: 517.987.  MSC: 37A50, 37B10.

\end{abstract}

\section{Introduction}

Finite  \ii{Mealy-type  automata} (and closely related to them 
\ii{Moore-type automata}) play an important  role in computer science.
Such  automata  also  play   remarkable  role  in  algebra,  dynamical  
systems, theory of random  walks, spectral theory of graphs, operator  
algebras,  holomorphic dynamics  and  other  areas  of  mathematics 
(see for instance \cite{auto} and references therein).

The main feature of an initial deterministic automaton $\cA_q$    
with a finite input alphabet $X$ and output  alphabet $Y$   is   that   it  
transforms  finite  words 
(strings)  over $X$  into  words    of  the  same  length over  $Y$, 
and  this transformation can  be extended  to a map $\hat{\cA_q}:X^\nn\to 
Y^\nn$ defined  on  the  space  $X^{\mathbb N}$ of  infinite 
words.  If  the  input  and  output  alphabets  coincide  then one  
can  iterate  the  map  $\hat{\mathcal A_q}$   which  leads   to   the  
dynamics on the  space  $X^*$  of  all  finite  words  over  $X$ as well as 
on  the  space $X^\nn$  of  infinite  words.  Also,  we  can  compose  
different maps  of  this  kind,  which  leads  to the automaton semigroups  
or  groups (in the  invertible  case).  The  space $X^\nn$  is endowed with 
the natural  product  topology  which makes  it  homeomorphic  to  a  
Cantor  set.    The  map $\hat{\mathcal A_q}$ is  continuous with  respect  
to  this  topology and  the  pair $(\hat{\mathcal A_q},X^\nn)$ 
is   a  topological  dynamical  system.  A  famous  example  of  this  
sort   is  given  by \ii{the  odometer} (see Figure \ref{fig:odom}).

\begin{figure}[t]
\centering
\includegraphics[width=0.5\textwidth]{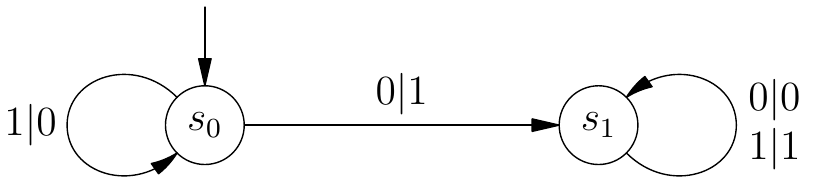}
\caption{The odometer, also known as the adding machine}
\label{fig:odom}
\end{figure}

The  map $\hat{\mathcal A_q}:X^\nn\to X^\nn$  can  
be  viewed  as transducer  that  transforms  the  input  sequence  of  
symbols into  output  sequence  and  it  can  be  considered  as  a coding  
map.  One  may  be  interested in  what  happens  statistically  with  the  
sequence  after  such  transformation.  For  instance,  what  happens with  
the  frequency  of  occurrence  of  a  fixed  symbol    from the  
alphabet?  Or what  happens  with  the  probability  distribution  on  the  
space  of  sequences   if the  input  sequence  is  random?

Ryabinin \cite{Ryab} (see also \cite{AleshKudr}, where more details  are  
given) raised the above  questions  in  the  case  
of  the  binary  alphabet  when different input symbols have identical and 
independent distribution.  In terms of ergodic theory, this  
means  that the  input  sequence  is generic with respect to a Bernoulli 
measure $\mu$.  It was  observed  that  the  output  sequence  is  in  
general  distributed  according  to a  different  law,  just  because  of  
the  change  of  the  frequencies  of  symbols.  Moreover, given the 
frequency $p$ of symbol $1$ in a $\mu$-generic sequence, a formula  was
suggested   for the   frequency $f(p)$ of  $1$ in a sequence generic with 
respect  to  the  image  measure  $(\hat{\mathcal A_q})_{\ast}\mu$.
The function $f$ was  suggested  to  be   called     a  \ii{stochastic  
function}.  No   justification of the formula  for $f(p)$ was  given.  A 
formula with a heuristic argument was presented in the book 
\cite{AleshKudr}.

The shift map $\sigma$ acts on sequences by deleting the  first symbol  and 
hence statistical properties of a $\mu$-generic sequence are impacted by 
ergodic 
properties of $\sigma$ relative to $\mu$.  While the  Bernoulli  measures 
are shift-invariant and ergodic, their images under $\hat{\mathcal A_q}$ 
are usually not.  The  nature  of  the  images  was thoroughly studied by 
Kravchenko \cite{Krav} in  the  case  of  the  alphabet  of  arbitrary 
cardinality and   a  formula  
for  the  frequencies of symbols in  the  output  sequences  was  found  
and  justified.  Moreover,  he showed  that  in  the  case  when  
automaton  $\mathcal A_q$  has  polynomially  bounded activity (as defined 
by S.~Sidki \cite{Sidki}),  the  image  measure  is  absolutely  
continuous  with  respect  to  $\mu$,  while   in  the case  of  a
\ii{strongly connected} automaton  it is typically singular with respect to 
$\mu$.  Singularity was proved by comparing the frequencies at which 
various symbols occur in a $\mu$-generic input sequence and the 
corresponding output sequence.

This paper  generalizes  and extends the  results  of  Kravchenko to the 
case when the measure $\mu$ belongs to a more general class of \emph{Markov 
measures}.  First we show  that  the  image $(\hat{\AA_q})_*\mu$ is  
absolutely continuous with respect to $\mu$ if the automaton $\AA_q$ is of 
low activity (Theorem \ref{polycase}).  This result is extended to 
non-invertible automata as well as to transformations generated by automata 
with infinitely many states.  We do have a restriction though.  Indeed, the 
Markov measures can have forbidden words, when certain transitions in the 
corresponding Markov chain have zero probability.  The automaton must not 
transform any allowed word into a forbidden one.  Our next result is a 
formula that expresses the frequencies  of  symbols in a sequence generic 
with respect  to  the image measure (Theorem \ref{mainthm}), which is then 
generalized to calculate frequencies of arbitrary words (Theorem 
\ref{mainthm1}).  In these two theorems, the automaton is supposed to be 
strictly connected.  The proof   uses the  idea of  converting the  Markov  
chain on the alphabet $X$ of the  automaton $\cA_q$ (given  by a 
stochastic  matrix $L$ that defines $\mu$) into  a  Markov  chain  on the  
product $S\times X$ (where  $S$  is a  set  of states  of $\cA_q$).  The  
latter is determined  by a stochastic matrix $T$ constructed as a  kind  of 
a  skew  product.  This 
allows to represent $(\hat{\AA_q})_*\mu$ as composed of pieces of a measure 
$Q$ that  is the  image  under  a  $1$-block factor map  of  a  Markov  
measure   on   $(S \times X)^\nn$.  Special  attention  is devoted to the 
case  when the  stationary probability vector $\bt$  of the matrix $T$ 
decomposes as a  tensor  product  of  the stationary probability vector 
$\bl$ of $L$ with another probability vector $\bk$.  This condition holds, 
e.g., when $\mu$ is a Bernoulli measure or the automaton is reversible.  For
our last result, under assumptions that $\bt=\bk\otimes\bl$ and the 
automaton is invertible and strongly connected, we prove that the Markov 
measure $\mu$ and its image $(\hat{\AA_q})_*\mu$ are either singular or the 
same (particular cases lead to Theorems \ref{sing-mainB} and 
\ref{sing-mainM}).

The Bernoulli and Markov measures belong to the class of finite-state (or 
self-similar) measures, which was considered by the authors in 
\cite{autolog}.  This looks like a natural class for future generalizations.

Now  let  us  discuss possible applications of the results obtained in this 
paper.  First  there  is very  interesting  group  theory  related to  
the finite automata  of  Mealy  type.  Namely, given a non-initial 
invertible automaton $\AA$, we can set any state $q$ as initial.  Hence the 
automaton generates several invertible transformations 
$\hat{\AA}_q:X^\nn\to X^\nn$, which in turn generate a transformation group 
$\GG(\AA)$.  Groups of this kind are called automaton (or self-similar) 
groups \cite{self-sim-groups}  and  they 
play an important  role  in  group  theory as they  were  used  to  solve a 
number  of  famous  problems  and  find  applications  in many  areas  of  
mathematics  \cite{self-sim-groups1}.  The  results  of  the present paper  
allow  for a deeper study of such  groups and  their  relation  to  the  
dynamics  and  information  theory.

Secondly, if the  automaton  $\AA$  is of polynomial activity,  then the 
Markov measures are quasi-invariant with respect to the group $\GG(\AA)$ 
and therefore can be  used  to  build a  Koopman  type  unitary  
representations in the  Hilbert  space $L^2(X^\nn,\mu)$.  Study  
of  such  representations was initiated  by A. Dudko  and  the  first  
author  in \cite{DudkoGri},  where  it  was  shown  that there  are  many 
pairwise disjoint  representations of  this  type,  they  are  irreducible  
and possess a number of interesting and  useful  properties.

\section{Preliminaries}

\subsection{The shift and the Markov measures}

Let $X$ be a finite set consisting of more than one element.  We refer to 
$X$ as the \emph{alphabet}.  Elements of $X$ are referred to as letters, 
symbols or characters.  Let $X^*$ denote the set of all finite strings 
$x_1x_2\ldots x_n$ of letters from $X$ (including the empty one 
$\varnothing$).  We refer to them as \emph{words} over $X$ and write 
without any delimiters.  Elements of $X$ are identified with one-letter 
words in $X^*$.  Let $X^\nn$ denote the set of all infinite 
sequences (or infinite words) $\om=\om_1\om_2\om_3 \ldots$ over $X$.  Given 
$u,w\in X^*$ and $\om\in X^\nn$, we can naturally define the concatenations 
$uw\in X^*$ and $u\om\in X^\nn$.  Then $u$ is called a \emph{prefix} of the 
word $uw$ and the sequence $u\om$.

The set $X^\nn$ is endowed with the product topology.  The topology is 
generated by the \emph{cylinders}, which are sets of the form 
$uX^\nn=\{u\om\mid \om\in X^\nn\}$, where $u\in X^*$.  The \textbf{shift} 
over the alphabet $X$ is a transformation $\sigma:X^\nn\to X^\nn$ given by 
$(\si(\om))_n=\om_{n+1}$ for all $\om\in X^\nn$ and $n\in\nn$.  The shift 
is continuous and non-invertible.  To simplify notation, we use the same 
symbol $\si$ even when dealing simultaneously with shifts over different 
alphabets.

Suppose $\mu$ is a Borel probability measure on $X^\nn$.  By Kolmogorov's  
theorem, $\mu$ is uniquely determined by its values on the cylinders.  
Conversely, for any function $f:X^*\to[0,\infty)$ satisfying 
$f(\varnothing)=1$ and $\sum_{x\in X}f(wx)=f(w)$ for all $w\in X^*$, there 
is a (unique) Borel probability measure $\nu$ on $X^\nn$ such that 
$\nu(wX^\nn)=f(w)$ for all $w\in X^*$.  The measure $\mu$ is 
\emph{shift-invariant} if $\mu(\si^{-1}(E))=E$ for any Borel set $E\subset 
X^\nn$.  A necessary and sufficient condition for this is that $\sum_{x\in 
X}\mu(xwX^\nn)=\mu(wX^\nn)$ for all $w\in X^*$.  The shift-invariant 
measure $\mu$ is \emph{ergodic} if any Borel set $E\subset X^\nn$ 
satisfying $\sigma^{-1}(E)=E$ has measure $0$ or $1$.

Let $Y$ be another alphabet and $g:X^\nn\to Y^\nn$ be a Borel measurable 
map.  Given a Borel probability measure $\mu$ on $X^\nn$, the 
\textbf{pushforward} of $\mu$ by $g$, denoted $g_*\mu$, is a Borel 
probability measure on $Y^\nn$ given by $g_*\mu(E)=\mu(g^{-1}(E))$ for all 
Borel sets $E\subset Y^\nn$.  If the map $g$ intertwines the shifts on 
$X^\nn$ and $Y^\nn$, that is, $g\sigma=\sigma g$, then the pushforward 
measure $g_*\mu$ is shift-invariant whenever $\mu$ is shift-invariant.  In 
the case $g$ is continuous, it satisfies $g\sigma=\sigma g$ if and only if 
there exist an integer $k\ge1$ and a function $\phi:X^k\to Y$ such that 
$(g(\om))_n=\phi(\om_n,\om_{n+1},\ldots,\om_{n+k-1})$ for all $\om\in 
X^\nn$ and $n\in\nn$.  Such a map is called a \textbf{$k$-block factor} 
map.  Note that the shift itself is a $2$-block factor map.

Any function $p:X\to \rr$ can be interpreted as a vector $p=(p_x)_{x\in X}$ 
which coordinates are indexed by symbols in $X$.  We use both $p(x)$ and 
$p_x$ as notation for the coordinates.  If the set $X$ is naturally 
ordered, we can write $p$ as a usual row vector.  The vector $p$ is a 
\emph{probability vector} if $p_x\ge0$ for all $x$ and $\sum_xp_x=1$.  The 
probability vector defines a \textbf{Bernoulli measure} $\mu$ on $X^\nn$ by 
$\mu(x_1x_2\ldots x_nX^\nn)=p_{x_1}p_{x_2}\dots p_{x_n}$ for any 
$x_1,x_2,\ldots,x_n\in X$.  Any Bernoulli measure is shift-invariant and 
ergodic.

Any function $L:X\times X\to \rr$ can be interpreted as a matrix 
$L=(L_{xy})_{x,y\in X}$ which rows and columns are indexed by symbols in 
$X$.  We use both $L(x,y)$ and $L_{xy}$ as notation for the entries.  If 
the set $X$ is naturally ordered, we can write $L$ as a usual matrix.  The 
matrix $L$ is \emph{stochastic} if all entries are nonnegative and 
$\sum_yL_{xy}=1$ for all $y$.  The stochastic matrix defines a Markov chain 
on $X$ such that $L_{xy}$ is the probability of transition from $x$ to 
$y$.  The stochastic matrix $L$ is called \emph{irreducible} if the Markov 
chain is irreducible, which means that for any $x,y\in X$ we can find 
$x_1=x,x_2,\ldots,x_n=y$ such that $L_{x_ix_{i+1}}>0$ for $1\le i\le n-1$.
Given a stochastic matrix $L=(L_{xy})_{x,y\in X}$ and a probability vector 
$\bl=(\bl_x)_{x\in X}$, we define a \textbf{Markov measure} $\mu$ on 
$X^\nn$ by
\[
\mu(x_1x_2x_3\ldots x_nX^\nn) = \bl_{x_1}L_{x_1x_2}L_{x_2x_3}\dots 
L_{x_{n-1}x_n}
\]
for any $x_1,x_2,x_3,\ldots,x_n\in X$.  The Bernoulli measures are a 
particular case of the Markov measures, when each row of the matrix $L$ 
coincides with $\bl$.  The Markov measure $\mu$ is shift-invariant if and 
only if $\bl$ is a \emph{stationary probability vector} of the matrix $L$, 
which means that $\bl L=\bl$, i.e., $\sum_x \bl_xL_{xy}=\bl_y$ for all 
$y$.  If, additionally, $L$ is irreducible then the Markov measure is 
ergodic.  Moreover, the stationary probability vector of an irreducible 
stochastic matrix is unique and positive.  For more details on Markov 
measures, see, e.g., \cite{Bill}.  In what follows we consider Markov 
measures that are shift-invariant but not necessarily ergodic.

Given a sequence $\om\in X^\nn$ and a letter $x\in X$, let $N(n)$ be the 
number of times $x$ occurs among the first $n$ terms of $\om$.  The limit 
of $N(n)/n$ as $n\to\infty$, if it exists, yields the \textbf{asymptotic 
frequency} at which $x$ occurs in the sequence $\om$.  We denote this limit 
by $\freq_\om(x)$.  If the limit does not exist then $\freq_\om(x)$ is not 
defined.  Let $\chi_{xX^\nn}$ be the characteristic function of the 
cylinder $xX^\nn$.  It is easy to see that
\[
\freq_\om(x)=\lim_{n\to\infty}\frac1n \sum_{i=0}^{n-1} 
\chi_{xX^\nn}(\sigma^i(\om)).
\]
Similarly, for any nonempty word $u\in X^*$ the limit
\[
\freq_\om(u)=\lim_{n\to\infty}\frac1n \sum_{i=0}^{n-1} 
\chi_{uX^\nn}(\sigma^i(\om)),
\]
if it exists, yields the asymptotic frequency at which $u$ occurs as a 
subword in $\om$ (compared with other words of the same length in $X^*$).

Suppose $\mu$ is a Borel probability measure on $X^\nn$.  If $\mu$ is 
shift-invariant then it follows from the Birkhoff ergodic theorem that 
$\freq_\om(u)$ is defined for $\mu$-almost all $\om\in X^\nn$.  If, 
additionally, $\mu$ is ergodic then $\freq_\om(u)=\mu(uX^\nn)$ for 
$\mu$-almost all $\om$. 

\subsection{The Mealy-Moore coding}

In this article we consider the Mealy automata, which are the simplest type 
of transducers with input and output (for a detailed exposition, see 
\cite{auto}).  By definition, a \textbf{Mealy automaton} (or simply an 
\textbf{automaton}) is a quadruple $\cA=(X,S,\pi,\la)$ consisting of two 
nonempty finite sets $X$ (the input/output alphabet) and $S$ (the set of 
states), and two functions, the transition function $\pi:S\times X\to S$ 
and the output function $\la:S\times X\to X$.  (One can consider a more 
general construction where the automaton has separate input and output 
alphabets $X$ and $Y$; then $\la$ takes values in $Y$.)  These functions 
are naturally extended to functions on $S\times X^*$ by 
$\pi(s,\varnothing)=s$, $\la(s,\varnothing)=\varnothing$, and recursive 
rules 
\begin{align*}
\pi(s, xw) &= \pi(s,x)\, \pi(\pi(s,x), w),\\
\lambda(s, xw) &= \lambda(s,x)\, \lambda(\pi(s,x), w),
\end{align*}
where $x\in X$ and $w\in X^*$.  The same recursive rules allow to extend 
$\pi$ and $\la$ to functions on $S\times X^\nn$, but this time $w\in X^\nn$.

Selecting a state $g\in S$ as initial makes $\cA$ into an \emph{initial 
automaton}.  The initial automaton generates transformations of $X^*$ and 
of $X^\nn$, both given by $w\mapsto \la(g,w)$ and referred to as the action 
of the state $g$ or, more generally, as an \textbf{automaton 
transformation}.  The state $g$ is \textbf{nontrivial} if the action is 
nontrivial.  Note that the action on $X^*$ uniquely determines the action 
on $X^\nn$, and vice versa.  By overloading notation, we use $g$ to denote 
either transformation.

All automaton transformations of $X^\nn$ are continuous.  An automaton with 
one state generates a $1$-block factor map.  No block factor map that is 
not a $1$-block factor map can be generated by an automaton.

Any automaton $\cA=(X,S,\pi,\la)$ can be pictured using its \textbf{Moore 
diagram}, which is a directed graph with labeled edges.  The vertices are 
the states and the edges correspond to transition routes (loops and 
multiple edges are possible).  Every edge carries a label consisting of two 
fields.  The top (or left) field is the input letter that invokes that 
particular transition.  The bottom (or right) field is the output letter 
produced during that.  Hence every edge is of the form $s \arr{x}{y} s'$ or 
$s \stackrel{x|y}{\to} s'$, where $\pi(s,x)=s'$ and $\la(s,x)=y$.  Multiple 
edges can be pictured as a single edge with multiple labels.  The 
action of a state $g$ on $X^*$ can be described using paths in the Moore 
diagram.  Namely, given an input word $x_1x_2\ldots x_n\in X^*$, we need to 
find a path of the form
\[
g \arr{x_1}{y_1} s_1 \arr{x_2}{y_2} \ldots \arr{x_n}{y_n} s_n.
\]
Such a path exists and is unique.  Then $g(x_1x_2\ldots x_n)= 
\la(g,x_1x_2\ldots x_n)=y_1y_2\ldots y_n$ and $\pi(g,x_1x_2\ldots x_n)=s_n$.
Likewise, the action of $g$ on $X^\nn$ can be described using infinite 
paths.

The automaton $\cA$ is called \textbf{strongly connected} if its Moore 
diagram is a strongly connected graph, which means that there is a path 
from any state to any other state.

The automaton $\cA=(X,S,\pi,\la)$ is called \textbf{invertible} if each 
state acts on $X$ by a permutation, that is, the function 
$\la(s,\cdot):X\to X$ is invertible for any $s\in S$.  Assuming this, let 
$\la'(s,x)$ be a unique letter such that $\la(s,\la'(s,x))=x$.  Also, let 
$\pi'(s,x)=\pi(s,\la'(s,x))$.  Then $\cA'=(X,S,\pi',\la')$ is called the 
\textbf{inverse automaton} of $\cA$.  In terms of the Moore diagrams, the 
automaton $\cA'$ is obtained from $\cA$ by interchanging the two fields of 
each label.  The action of any state $s\in S$ on $X^*$ (or on $X^\nn$) 
generated by $\cA'$ is the inverse of the action of the same state 
generated by $\cA$.  It follows that the automaton is invertible if and 
only if the action of each state on $X^*$ (or on $X^\nn$) is invertible.  
In the case the automaton is strongly connected, it is enough to know that 
the action of one state is invertible.

\subsection{Endomorphisms of a regular rooted tree}

Given an alphabet $X$, let $\cT(X)$ be a graph with the vertex set $X^*$ in 
which two vertices are connected by an edge if and only if one of them is 
obtained by adding one letter at the end of the other.  Then $\cT(X)$ is an 
\emph{$m$-regular rooted tree}, where $m=|X|$, the number of letters in 
$X$.  The root is the empty word.  All words of a fixed length $k$ form 
the \emph{$k$-th level} of the tree as they are at distance $k$ from the 
root.  An invertible map $g:X^*\to X^*$ is an \textbf{automorphism} of the 
tree $\cT(X)$ if it maps adjacent vertices to adjacent vertices.  Any 
automorphism fixes the root and hence preserves each level of $\cT(X)$.  An 
arbitrary map $g:X^*\to X^*$ is called an \textbf{endomorphism} of the tree 
$\cT(X)$ if it maps adjacent vertices to adjacent vertices and also 
preserves each level.  An equivalent condition is that $g$ preserves the 
length of any word and does not decrease the length of the longest common 
prefix of any two words.  In particular, any automaton transformation of 
$X^*$ is a tree endomorphism.

The set $X^\nn$ of infinite sequences is naturally identified with the 
\emph{boundary} of the rooted tree $\cT(X)$, which consists of infinite 
paths without backtracking that start at the root.  Consequently, any tree 
endomorphism $h:X^*\to X^*$ induces a unique transformation $\tilde h: 
X^\nn\to X^\nn$ such that $h(u)$ is a prefix of $\tilde h(\om)$ whenever 
$u\in X^*$ is a prefix of $\om\in X^\nn$.  If $h$ is an automaton 
transformation, then $\tilde h$ is generated by the same initial 
automaton.  Note that $\tilde h$ does not decrease the length of the 
longest common prefix of any two sequences.  Moreover, any transformation 
of $X^\nn$ with the latter property is induced by a unique tree 
endomorphism.  In view of this, we refer to $\tilde h$ itself as a tree 
endomorphism and also as the action of $h$ on $X^\nn$.

Given a tree endomorphism $g:X^*\to X^*$, for any word $u\in X^*$ there 
exists a unique map $g|_u:X^*\to X^*$ such that $g(uw)=g(u)\,g|_u(w)$ for 
all $w\in X^*$.  The map $g|_u$, which is also a tree endomorphism, is 
called the \textbf{restriction} (or \textbf{section}) of $g$ by the word 
$u$.  The restriction $g|_u$ describes how $g$ acts inside a subtree of 
$\cT(X)$ with the vertex set $uX^*=\{uw\mid w\in X^*\}$, which is 
canonically isomorphic to the entire tree.  Likewise, we can define 
restrictions for a tree endomorphism $g:X^\nn\to X^\nn$ (but this time 
$w\in X^\nn$).  If a tree endomorphism $g$ is generated by an automaton 
$\cA=(X,S,\pi,\la)$ with initial state $g$, then any restriction $g|_u$ is 
the action of another state of the same automaton, namely, $\pi(g,u)$.  In 
the case the automaton is strongly connected, all states are restrictions 
of one another.

A tree endomorphism is called \bb{finite-state} if it has only finitely 
many distinct restrictions.  Given a finite-state tree endomorphism 
$g:X^*\to X^*$, we associate to it the \emph{automaton of restrictions} 
$\cA=(X,S,\pi,\la)$, where $S=\{g|_w: w \in X^*\}$, $\pi(s,x)=s|_x$ and 
$\lambda(s, x)=s(x)$ for all $s\in S$ and $x\in X$.  Then $g$ is generated 
by $\cA$ with initial state $g$.  An arbitrary endomorphism of $\cT(X)$ 
could be similarly generated by the automaton of restrictions if we allowed 
automata with infinitely many states (see \cite{auto}), which we do not.

\section{Tree endomorphisms of low activity}\label{poly}

Suppose a transformation $g:X^*\to X^*$ is an endomorphism of the regular 
rooted tree $\cT(X)$.  For any integer $n\ge0$ let $R_g(n)$ denote the 
number of words $w\in X^*$ of length $n$ such that the restriction $g|_w$ 
is nontrivial (i.e., not the identity map).  The function $R_g$ describes 
the \bb{activity growth} of $g$ as the length of input increases.  It is 
not uncommon that only few (if any) restrictions of $g$ are trivial, in 
which case $R_g(n)$ grows exponentially in $n$, i.e., $R_g(n)\ge c^n$ for 
some $c>1$ and all sufficiently large $n$.  For example, if $g$ is 
generated by an automaton with no trivial state, then all restrictions are 
nontrivial so that $R_g(n)=|X|^n$ for any $n$.  However in this section we 
are looking for transformations with much slower activity growth. 

We say that the endomorphism $g$ is of \bb{polynomial activity growth} (or 
simply of \bb{polynomial activity}) if the function $R_g(n)$ grows at most 
polynomially in $n$, that is, $R_g(n)\le cn^\alpha$ for some $c,\alpha>0$ 
and all $n$.  Similarly, we can consider endomorphisms $g$ of \bb{bounded 
activity}, when the function $R_g$ is bounded (they form a smaller class), 
and of \bb{subexponential activity growth}, when $R_g(n)\le c^n$ for any 
fixed $c>1$ and all sufficiently large $n$ (those form a larger class).  
Note that $(g|_u)|_w=g_{uw}$ for all $u,w\in X^*$.  Therefore 
$R_{g|_u}(n)\le R_g(n+k)$, where $k$ is the length of $u$.  It follows that 
all three classes are closed under taking restrictions.

If a restriction $g|_u$ is nontrivial, then so is the restriction of $g$ by 
any prefix of $u$.  As a consequence, $R_g(n+1)\le |X|R_g(n)$ for all $n$.  
Conversely, if a function $f:\nn\cup\{0\}\to\nn\cup\{0\}$ satisfies 
$f(0)\le1$ and $f(n+1)\le |X|f(n)$ for all $n\ge0$, then $f$ is the 
activity growth function of some tree endomorphism.  Hence various tree 
endomorphisms exhibit a huge variety of activity growths including 
intermediate between polynomial and exponential.  As there are only 
countably many finite-state tree endomorphisms, their activity growth 
cannot be so diverse.  In fact, any finite-state endomorphism has either 
polynomial or exponential activity growth.  There is an elegant criterion, 
due to Sidki \cite{Sidki} who introduced the notion of activity growth, 
that allows to distinguish between these two possibilities.

\begin{proposition}\label{notwocycle}
All tree endomorphisms generated by an automaton $\cA$ have polynomial 
activity growth if and only if the Moore diagram of $\cA$ does not admit 
two distinct simple cycles through any nontrivial state.
\end{proposition}

Given a tree endomorphism $g:X^*\to X^*$, let us associate to it two sets 
of finite words.  The set $V(g)$ consists of all $w\in X^*$ such that the 
restriction $g|_u$ is trivial whenever $g(u)=w$.  This includes a 
possibility that no such words $u$ exist.  If $g$ is invertible, then $w\in 
V(g)$ if and only if $g^{-1}|_w$ is trivial.  The set $V_{\max}(g)$ is a 
subset of $V(g)$.  A word $w\in V(g)$ belongs to $V_{\max}(g)$ if no word 
in $V(g)$ is a proper prefix of $w$.

In the case $g$ is invertible, it is an automorphism of the regular rooted 
tree $\cT(X)$, and so is the inverse $g^{-1}$.  In this case, any word 
$w\in V(g)$ corresponds to a subtree $wX^*$ such that the action of 
$g^{-1}$ inside $wX^*$ is trivial.  Words in $V_{\max}(g)$ correspond to 
maximal subtrees of that kind.

Now we turn to the action of $g$ on $X^\nn$.  By definition of the set 
$V_{\max}(g)$, the cylinders $wX^\nn$, $w\in V_{\max}(g)$ are disjoint 
subsets of $X^\nn$.  Let us consider the complement
\[
\Om_g=X^\nn\setminus\bigcup_{w\in V_{\max}(g)}wX^\nn.
\]
The size of the set $\Om_g$ depends on the activity growth of $g$. 

\begin{lemma}\label{polysupport}
Suppose $g:X^\nn\to X^\nn$ is a finite-state tree endomorphism of 
polynomial activity.  Then the sets $\Om_g$ and $g^{-1}(\Om_g)$ are at most 
countable.
\end{lemma}

\begin{proof}
Let $\cA=(X,S,\pi,\la)$ be the automaton of restrictions of $g$.  Then 
every state of $\cA$ is of polynomial activity.  By Proposition 
\ref{notwocycle}, the Moore diagram of $\cA$ does not admit two distinct 
simple cycles through any nontrivial state.

Let $\Om'_g$ denote the set of all sequences $\om\in X^\nn$ such that the 
restriction of $g$ by any prefix of $\om$ is nontrivial.  Given 
$\om=x_1x_2x_3\ldots\in \Om'_g$, consider a sequence of states 
$s_0,s_1,s_2,\dots$ that the automaton $\cA$ with initial state $g$ goes 
through while processing the input $\om$.  We have $s_0=g$ and 
$s_n=\pi(s_{n-1},x_n)$ for $n\ge1$.  Each $s_n$ is nontrivial since 
$\om\in\Om'_g$.  As there are only finitely many states, some $s\in S$ is 
visited infinitely often.  If $s_k=s_n=s$ for some $k$ and $n$, $k<n$, then 
$s|_{x_{k+1}x_{k+2}\ldots x_n}=s$.  Let $u$ be the shortest nonempty word 
in $X^*$ such that $s|_u=s$.  Since the Moore diagram of $\cA$ does not 
admit two distinct simple cycles through any nontrivial state, it follows 
that $s|_w=s$ if and only if the word $w$ is obtained by repeating $u$ 
several times.  We conclude that some tail of the sequence $\om$ coincides 
with the periodic sequence $uuu\ldots$ so that $\om$ is eventually 
periodic.  As there are only countably many eventually periodic sequences 
in $X^\nn$, the set $\Om'_g$ is at most countable.

Next we show that $\Om_g\subset g(\Om'_g)$, which will imply that $\Om_g$ 
is also at most countable.  Indeed, let $\om=x_1x_2x_3\ldots$ be in 
$\Om_g$.  Then no prefix $x_1x_2\ldots x_n$ of $\om$ belongs to $V(g)$.  
Hence there is a word $u^{(n)}$ of length $n$ such that $g(u^{(n)})= 
x_1x_2\ldots x_n$ and the restriction $g|_{u^{(n)}}$ is nontrivial.  Since 
$X$ is a finite set, we can build inductively a sequence $\om'\in X^\nn$ 
such that any prefix of $\om'$ is also a prefix for infinitely many words 
$u^{(n)}$.  If a word $w$ occurs as a prefix for another word $u$, then 
$g(w)$ is a prefix for $g(u)$ and $g|_w$ is nontrivial whenever $g|_u$ is 
nontrivial.  It follows that $g(\om')=\om$ and $\om'\in\Om'_g$ so that 
$\om\in g(\Om'_g)$.

Let $\om\in\Om_g$ and suppose $\om'$ is a pre-image of $\om$ under the 
transformation $g$.  If $\om'$ is not in $\Om'_g$ then there is a prefix 
$u$ of $\om'$ such that the restriction $g|_u$ is trivial.  This implies 
that $g$ does not change the tail of $\om'$ following the prefix $u$.  
Hence $\om$ can be obtained from $\om'$ by changing some letters in the 
prefix $u$.  We conclude that any element of $g^{-1}(\Om_g)\setminus\Om'_g$ 
coincides with some element of $\Om_g$ up to finitely many terms.  Note 
that for any $\om\in X^\nn$ there are only countably many sequences in 
$X^\nn$ that coincide with $\om$ up to finitely many terms.  Since the sets 
$\Om_g$ and $\Om'_g$ are at most countable, it follows that $g^{-1}(\Om_g)$ 
is at most countable as well.
\end{proof}

A Borel measure on $X^\nn$ is called \textbf{non-atomic} if every 
one-element set has measure zero.  Under the assumptions of Lemma 
\ref{polysupport}, we have $\mu(\Om_g)=\mu(g^{-1}(\Om_g))=0$ for any 
non-atomic measure $\mu$.

\begin{lemma}\label{subexpsupport}
If $g:X^\nn\to X^\nn$ is a tree endomorphism of subexponential activity 
growth, then $\mu(\Om_g)=\mu(g^{-1}(\Om_g))=0$ for any non-atomic Markov 
measure $\mu$ on $X^\nn$.
\end{lemma}

\begin{proof}
Let $\mu$ be an arbitrary non-atomic Markov measure on $X^\nn$.  It is 
defined by a stochastic matrix $L$ with stationary probability vector 
$\bl$.  First we need to estimate measures of cylinders.  Let $\alpha_0$ be 
the largest entry of $L$ different from $1$.  Let $m$ be the number of 
letters in $X$.  We claim that $\mu(wX^\nn)\le c\alpha^n$ for any word 
$w\in X^*$ of length $n\ge1$, where $c=\alpha_0^{-1-1/m}$ and  
$\alpha=\alpha_0^{1/m}$ (note that $\alpha<1$).  Assume the contrary:  
$\mu(wX^\nn)>c\alpha^n$ for some $w=x_1x_2\ldots x_n$, where each $x_i\in 
X$.  The measure is given by $\mu(wX^\nn)=\bl_{x_1}L_{x_1x_2}\dots 
L_{x_{n-1}x_n}$, where no factor in the product exceeds $1$.  Since 
$c\alpha^n=\alpha_0^{(n-1)/m-1}$, the sequence $L_{x_1x_2},L_{x_2x_3}, 
\dots, L_{x_{n-1}x_n}$ contains no more than $(n-1)/m-1$ numbers different 
from $1$.  Also, $n-1>m$ since $\alpha_0^{(n-1)/m-1}<\mu(wX^\nn)\le1$.  It 
follows that the sequence admits $m$ consecutive $1$s.  That is,
$L_{x_ix_{i+1}}=1$ for $k\le i\le k+m-1$, where $1\le k\le n-m$.  Then some 
letter $x\in X$ occurs more than once in the word $x_kx_{k+1}\ldots 
x_{k+m}$, that is, $x_j=x_{j'}=x$ for some $j$ and $j'$, $k\le j<j'\le 
k+m$.  Let us take the first $j$ letters of $w$ and append to them the word 
$x_{j+1}x_{j+2}\ldots x_{j'}$ repeated infinitely many times.  We obtain an 
infinite sequence $\om=y_1y_2y_3\ldots$ in $X^\nn$.  By construction, 
$y_i=x_i$ for $1\le i\le j$ and $L_{y_iy_{i+1}}=1$ for $i\ge j$.  As a 
consequence, a cylinder $y_1y_2\ldots y_iX^\nn$ has the same measure 
$M=\bl_{x_1}L_{x_1x_2}\dots L_{x_{j-1}x_j}$ for all $i\ge j$.  This measure 
is not zero since $\mu(wX^\nn)$ is not zero.  The cylinders $y_1y_2\ldots 
y_iX^\nn$ are nested and their intersection is $\{\om\}$.  It follows that
$\mu(\{\om\})=M\ne 0$, which contradicts with $\mu$ being a non-atomic 
measure.

For any $n\ge1$ let $W'_n$ denote the set of all words $w\in X^*$ of length 
$n$ such that the restriction $g|_w$ is nontrivial.  Further, let 
$W_n=g(W'_n)$.  For any $k$, $0\le k\le n$, let $W_{n,k}$ be the set of all 
words of length $n$ that coincide with a word in $W_n$ up to changing some 
of the first $k$ letters.  The cardinality of the set $W'_n$ is 
$|W'_n|=R_g(n)$.  Then $|W_n|\le |W'_n|=R_g(n)$ and $|W_{n,k}|\le 
|X|^k|W_n|\le m^kR_g(n)$.

Let $\Xi'_n$ be the union of cylinders $wX^\nn$ over all $w\in W'_n$, let 
$\Xi_n$ be the union of $wX^\nn$ over all $w\in W_n$, and let $\Xi_{n,k}$ 
be the union of $wX^\nn$ over all $w\in W_{n,k}$.  Since $\mu(wX^\nn)\le 
c\alpha^n$ for any word $w$ of length $n$, it follows that $\mu(\Xi'_n)\le 
c\alpha^nR_g(n)$, $\mu(\Xi_n)\le c\alpha^nR_g(n)$ and $\mu(\Xi_{n,k})\le 
c\alpha^nm^kR_g(n)$.  By assumption, $R_g(n)$ grows subexponentially in 
$n$.  Since $\alpha<1$, we conclude that $\mu(\Xi'_n)$, $\mu(\Xi_n)$ and 
$\mu(\Xi_{n,k})$ all tend to $0$ as $n\to\infty$.

Just like in the proof of Lemma \ref{polysupport}, consider the set 
$\Om'_g$ of all sequences $\om\in X^\nn$ such that the restriction of $g$ 
by any prefix of $\om$ is nontrivial.  Clearly, $\Om'_g\subset\Xi'_n$ for 
all $n$.  Since $\mu(\Xi'_n)\to0$ as $n\to\infty$, the set $\Om'_g$ has 
measure zero.  Just like in the proof of Lemma \ref{polysupport}, we can 
show that $\Om_g\subset g(\Om'_g)$.  Then $\Om_g\subset\Xi_n$ for all $n$.  
Since $\mu(\Xi_n)\to0$ as $n\to\infty$, the set $\Om_g$ has measure zero.  
Further, we can show that any pre-image under $g$ of any $\om\in\Om_g$ 
either belongs to $\Om'_g$ or coincides with $\om$ up to finitely many 
terms.  Hence any element of $g^{-1}(\Om_g)\setminus\Om'_g$ belongs to 
$\Xi_{n,k}$ for some $k$ (depending on the element) and all $n$.  Since 
$\mu(\Xi_{n,k})\to0$ as $n\to\infty$ for any fixed $k$, it follows that 
$\mu(g^{-1}(\Om_g)\setminus\Om'_g)=0$.  We already know that 
$\mu(\Om'_g)=0$.  Thus $\mu(g^{-1}(\Om_g))=0$.
\end{proof}

Lemmas \ref{polysupport} and \ref{subexpsupport} suggest that a tree 
endomorphism of slow activity growth changes only finitely many terms in a 
generic infinite sequence $\om\in X^\nn$.  This observation leads to the 
following result.

\begin{theorem}\label{polycase}
Let $\mu$ be a non-atomic Markov measure on $X^\nn$ and $g:X^\nn\to X^\nn$ 
be a tree endomorphism of subexponential activity growth.  Then the measure 
$g_*\mu$ is absolutely continuous with respect to $\mu$ if and only if 
$\mu(wxX^\nn)=0$ implies $\mu(g^{-1}(wxX^\nn))=0$ for all $w\in 
V_{\max}(g)$ and $x\in X$.  If this is the case, then the Radon-Nikodym 
derivative is given by
\begin{equation}\label{eq:R-N}
\frac{dg_*\mu}{d\mu} = \sum_{
\substack{w\in V_{\max}(g),\, x \in X:\\ \mu(wxX^\nn)\ne0}}
\frac{\mu(g^{-1}(wxX^\nn))}{\mu(wxX^\nn)}\,\chi_{wxX^\nn}.
\end{equation}
\end{theorem}

\begin{proof}
If the measure $g_*\mu$ is absolutely continuous with respect to $\mu$, 
then $\mu(E)=0$ implies $\mu(g^{-1}(E))=g_*\mu(E)=0$ for any measurable set 
$E\subset X^\nn$.  Hence the conditions of the theorem are clearly 
necessary.  Now assume they hold.  We need to show that $g_*\mu=D\mu$, where
the function $D:X^\nn\to\rr$ is given by \eqref{eq:R-N}.

Consider any $w\in V_{\max}(g)$ and $x\in X$ such that $\mu(wxX^\nn)\ne0$. 
Let $\delta_{w,x}=\mu(g^{-1}(wxX^\nn))/\mu(wxX^\nn)$.  First we are going 
to show that $g_*\mu(C)=\delta_{w,x}\mu(C)$ for any cylinder $C\subset 
wxX^\nn$.  The cylinder $C$ is of the form $wxw'X^\nn$, where $w'\in X^*$.
Let $U_w$ be the set of all words $u\in X^*$ such that $g(u)=w$.  The 
pre-image $g^{-1}(wX^\nn)$ is the disjoint union of cylinders $uX^\nn$, 
$u\in U_w$.  Since the restriction $g|_u$ is trivial for each $u\in U_w$, 
it follows that $g^{-1}(wxX^\nn)$ is the union of cylinders $uxX^\nn$, 
$u\in U_w$, while $g^{-1}(C)$ is the union of cylinders $uxw'X^\nn$, $u\in 
U_w$.  Let $w=x_1x_2\ldots x_n$ and $w'=x'_1x'_2\ldots x'_k$ ($x_i,x'_j\in 
X$).  The Markov measure $\mu$ is defined by a stochastic matrix $L$ with 
stationary probability vector $\bl$.  We have
\begin{align*}
\mu(wxX^\nn)&=\bl_{x_1}L_{x_1x_2}\dots L_{x_{n-1}x_n}L_{x_nx},\\
\mu(C)&=\bl_{x_1}L_{x_1x_2}\dots L_{x_{n-1}x_n}L_{x_nx}L_{xx'_1}
L_{x'_1x'_2}\dots L_{x'_{k-1}x'_k},
\end{align*}
which implies that $\mu(C)=\mu(wxX^\nn)L_{xx'_1}L_{x'_1x'_2}\dots 
L_{x'_{k-1}x'_k}$.  Similarly,
\[
\mu(uxw'X^\nn)=\mu(uxX^\nn)L_{xx'_1}L_{x'_1x'_2}\dots L_{x'_{k-1}x'_k}
\]
for all words $u\in X^*$.  Summing up the latter equality over $u\in U_w$, 
we obtain
\[
\mu(g^{-1}(C))=\mu(g^{-1}(wxX^\nn))L_{xx'_1}L_{x'_1x'_2}\dots 
L_{x'_{k-1}x'_k}.
\]
It follows that
\[
g_*\mu(C)=\mu(g^{-1}(C))=\delta_{w,x}\mu(wxX^\nn)L_{xx'_1}L_{x'_1x'_2}\dots 
L_{x'_{k-1}x'_k}=\delta_{w,x}\mu(C).
\]

To prove that $g_*\mu=D\mu$, it is enough to show that the two measures 
agree on all cylinders.  Take any cylinder $C\subset X^\nn$.  The set 
$X^\nn$ is the disjoint union of $\Om_g$ and all cylinders of the 
form $wxX^\nn$, where $w\in V_{\max}(g)$ and $x\in X$.  By definition, the 
function $D$ takes a constant value on each $wxX^\nn$, which is $\de_{w,x}$ 
if $\mu(wxX^\nn)\ne0$ and $0$ otherwise.  Also, $D$ is zero on $\Om_g$.  It 
follows that
\[
\int_C D(\om)\,d\mu(\om)=\sum_{
\substack{w\in V_{\max}(g),\, x \in X:\\ \mu(wxX^\nn)\ne0}}
\delta_{w,x}\mu(C\cap wxX^\nn).
\]
Since $C$ is a cylinder, the intersection $C\cap wxX^\nn$ is either a 
cylinder or the empty set.  By the above, $\delta_{w,x}\mu(C\cap wxX^\nn)= 
g_*\mu(C\cap wxX^\nn)$.  Further, if $\mu(wxX^\nn)=0$ for some $w\in 
V_{\max}(g)$ and $x\in X$, then $g_*\mu(wxX^\nn)=0$ by assumption.  As a 
consequence, $g_*\mu(C\cap wxX^\nn)=0$.  Finally, $g_*\mu(\Om_g)=0$ due to 
Lemma \ref{subexpsupport}.  Hence $g_*\mu(C\cap\Om_g)=0$.  We conclude that
\[
\int_C D(\om)\,d\mu(\om)=g_*\mu(C\cap\Om_g)+
\sum_{w\in V_{\max}(g),\,x\in X} g_*\mu(C\cap wxX^\nn) =g_*\mu(C),
\]
which completes the proof.
\end{proof}

The set $V_{\max}(g)$ is rarely finite.  Therefore the conditions of 
Theorem \ref{polycase} might not be easy to verify, especially if $g$ is 
not invertible.  We can replace them with simpler but somewhat stronger 
conditions.

\begin{corollary}
Let $\mu$ be a Markov measure on $X^\nn$ defined by a stochastic matrix $L$ 
with stationary vector $\bl$, and $g:X^\nn\to X^\nn$ be a tree endomorphism 
of subexponential activity growth.  If all coordinates of $\bl$ and all 
entries of $L$ are positive, then the measure $g_*\mu$ is absolutely 
continuous with respect to $\mu$, with the Radon-Nikodym derivative given by
\[
\frac{dg_*\mu}{d\mu} = \sum_{w\in V_{\max}(g),\, x \in X}
\frac{\mu(g^{-1}(wxX^\nn))}{\mu(wxX^\nn)}\,\chi_{wxX^\nn}.
\]
\end{corollary}

\begin{proof}
Since all coordinates of $\bl$ and all entries of $L$ are positive, every 
cylinder has nonzero measure.  Besides, all entries of $L$ are less than 
$1$, which implies that the measure $\mu$ is non-atomic.  It remains to 
apply Theorem \ref{polycase}.
\end{proof}

\begin{corollary}
Let $\mu$ be a non-atomic Markov measure on $X^\nn$ defined by a stochastic 
matrix $L$ with stationary vector $\bl$, and $g:X^\nn\to X^\nn$ be a tree 
endomorphism of polynomial activity generated by an automaton 
$\cA=(X,S,\pi,\la)$.  Suppose that $\bl_x=0$ whenever $\bl_{\la(g,x)}=0$ 
and $L_{x,y}=0$ whenever $L_{\la(s,x),\,\la(\pi(s,x),y)}=0$ (for all $s\in 
S$ and $x,y\in X$).  Then the measure $g_*\mu$ is absolutely continuous 
with respect to $\mu$, with the Radon-Nikodym derivative given by 
\eqref{eq:R-N}.
\end{corollary}

\begin{proof}
In view of Theorem \ref{polycase}, we only need to show that 
$\mu(wxX^\nn)=0$ implies $\mu(g^{-1}(wxX^\nn))=0$ for all $w\in 
V_{\max}(g)$ and $x\in X$.  We are going to show more, namely, 
$\mu(wX^\nn)=0$ implies $\mu(g^{-1}(wX^\nn))=0$ for all $w\in X^*$.  Since 
the pre-image $g^{-1}(wX^\nn)$ is the union of cylinders $uX^\nn$ over all 
words $u$ such that $g(u)=w$, it is enough to show that $\mu(wX^\nn)=0$ and 
$g(u)=w$ implies $\mu(uX^\nn)=0$.

Suppose $w=x_1x_2\ldots x_n$ and $u=y_1y_2\ldots y_n$ are words of length 
$n\ge1$ such that $g(u)=w$.  We have $\mu(wX^\nn)=\bl_{x_1}L_{x_1x_2}\dots 
L_{x_{n-1}x_n}$ and $\mu(uX^\nn)=\bl_{y_1}L_{y_1y_2}\dots L_{y_{n-1}y_n}$.  
Let $s_1=g$ and $s_i=\pi(g,y_1y_2\ldots y_{i-1})$ for $2\le i\le n$.  Then 
$\la(s_i,y_i)=x_i$ for $1\le i\le n$ and $\pi(s_i,y_i)=s_{i+1}$ for $1\le 
i\le n-1$.  It follows that $\bl_{x_1}=0$ implies $\bl_{y_1}=0$ and 
$L_{x_ix_{i+1}}=0$ implies $L_{y_iy_{i+1}}=0$ for any $i$, $1\le i\le n-1$.
Thus $\mu(wX^\nn)=0$ implies $\mu(uX^\nn)=0$.
\end{proof}

\section{Strongly connected automata}\label{strong}

Suppose $\mu$ is a shift-invariant, ergodic Markov measure on $X^\nn$ 
defined by an irreducible stochastic matrix $L$ with stationary vector 
$\bl$.  Let $g:X^\nn\to X^\nn$ be an automaton transformation and $\om$ be 
a $\mu$-generic sequence in $X^\nn$.  To learn about statistical properties 
of the sequence $g(\om)$, we should study the pushforward measure $g_*\mu$.
Unfortunately, the measure $g_*\mu$ need not be shift-invariant, let alone 
ergodic.  On the cylinders, it is given by
\[
g_*\mu(y_1y_2\ldots y_n X^\nn) = \sum_{g \arr{x_1}{y_1} s_1 \arr{x_2}{y_2} 
\ldots \arr{x_n}{y_n} s_n} \bl(x_1)L(x_1,x_2)\ldots L(x_{n-1},x_n),
\]
where the sum is over all paths of the form $g \arr{x_1}{y_1} s_1 
\arr{x_2}{y_2} \ldots \arr{x_n}{y_n} s_n$ in the Moore diagram of the 
automaton $\cA=(X,S,\pi,\la)$ generating $g$.

One way to address this difficulty is to keep track of the states $\cA$ 
goes through along with the output.  We are going to define maps to obtain 
a commutative diagram
\[
\begin{diagram}
& & (S\times X)^\nn & & \\
& \ruInto^{\pig} & & \rdOnto^{\tilde\lambda} & \\
X^\nn & & \rTo^{g} & & g(X^\nn)\subset X^\nn \\
\end{diagram}
\]
so that $g_*\mu = \lags\pigs\mu$.  Then we introduce a shift-invariant 
measure $Q$ on $X^\nn$ that is the pushforward of a Markov measure under 
the 1-block factor map $\lag$.  Under some assumptions on the automaton 
$\cA$ and the matrix $L$, the measure $Q$ is ergodic while the measure 
$g_*\mu$ is absolutely continuous with respect to $Q$.

Let us begin with defining a map ${\tilde\pi : S \times X^\nn \to (S\times 
X)^\nn}$ recursively by
\[
 \tilde\pi(s, x\om) = (s, x)\, \tilde\pi(\pi(s, x),\om)
\]
for all $s\in S$, $x\in X$ and $\om\in X^\nn$.  Then for any state $s\in S$ 
we define a map $\tilde\pi_s:X^\nn\to(S\times X)^\nn$ by 
$\tilde\pi_s(\om)=\tilde\pi(s,\om)$, $\om\in X^\nn$.

Recall that $g$ is one of the states of the automaton $\cA$.  Given a 
sequence $\tilde\pi_g(\om)=\tilde\pi(g,\om)\in(S\times X)^\nn$, we can 
extract the output $g(\om)=\lambda(g,\om)$ simply by looking at the 
states.  Hence we define a map $\tilde\lambda:(S\times X)^\nn\to X^\nn$ 
recursively by
\[
 \tilde\lambda \bigl((s, x)\tilde\om\bigr) = \lambda(s, x) 
 \tilde\lambda(\tilde{\om})
\]
for all $s\in S$, $x\in X$ and $\tilde{\om}\in (S\times X)^\nn$.  Note that 
$\tilde{\lambda}$ is a $1$-block factor map.  Now for every infinite path
\[
 g\arr{x_1}{y_1} s_1\arr{x_2}{y_2} \ldots \arr{x_n}{y_n} 
 s_n\arr{x_{n+1}}{y_{n+1}} \ldots
\]
in the Moore diagram of the automaton $\cA$ we have
\begin{align*}
 \tilde\pi_{g} (x_1x_2\ldots x_n\ldots) &= (g, x_1)(s_1, x_2) 
 \ldots(s_{n-1},x_n)\ldots,\\
 \tilde\lambda\bigl( (s_0, x_1)(s_1, x_2)\ldots(s_{n-1},x_n)\ldots\bigr) &= 
 y_1y_2\ldots y_n\ldots
\end{align*}
In particular, $\tilde\lambda(\tilde\pi_g(\om))=g(\om)$ for all $\om\in 
X^\nn$.  Hence we do have the commutative diagram.

The measure $\pigs\mu$ on $(S\times X)^\nn$ is easier to treat than the 
measure $g_*\mu$ on $X^\nn$.  On the cylinders, the former is given by
\[
\pigs\mu\bigl((g, x_1)(s_1, x_2)\ldots(s_{n-1}, x_n)(S\times X)^\nn\bigr)=
 \bl_{x_1}L_{x_1x_2}\ldots L_{x_{n-1}{x_n}}
\]
if \,$g \arr{x_1}{y_1} s_1 \arr{x_2}{y_2} \ldots \arr{x_n}{y_n} s_n$\, is a 
valid path for some $y_1,y_2,\dots,y_n\in X$ and $s_n\in S$.  Otherwise the 
cylinder has measure $0$.

In a sense, the measure $\pigs\mu$ is ``piecewise'' Markov, scaled by 
constants on cylinders $(s, x)(S \times X)^\nn$.  To make this statement 
more precise, let us introduce a matrix $T = T_{L,\cA}$ with rows and 
columns indexed by elements of $S \times X$, and entries given by
\begin{equation}\label{eq:T}
T_{(s_0,x_0)(s_1,x_1)} = \left\{\!\begin{array}{cl}
                  L(x_0, x_1) & \mbox{if } \,\pi(s_0, x_0) = s_1,\\
                  0 & \mbox{otherwise}.
                 \end{array}\right.
\end{equation}
This matrix is stochastic.  Indeed,
\[
 \sum_{(r, y)}T_{(s,x)(r,y)} = \sum_{y} L(x,y) = 1
\]
since $T_{(s,x)(r,y)} \neq 0$ for at most one choice of $r$, $r=\pi(s,x)$.

Let $\bt$ be a stationary probability vector of the stochastic matrix $T$.  
Recall that $\bt$ is a row vector which coordinates are indexed by elements 
of $S\times X$.  All coordinates are nonnegative and add up to $1$.  The 
vector $\bt$ satisfies the matrix identity $\bt T=\bt$.  Let $P$ denote the 
Markov measure on $(S\times X)^\nn$ with transition matrix $T$ and initial 
probability distribution $\bt$.  On the cylinders, the measure $P$ is given 
by
\begin{align*}
P\bigl((s_0, x_0)\ldots(s_{n}, x_n) (S\times X)^\nn\bigr) &
 =\bt_{(s_0,x_0)} T_{(s_0, x_0)(s_0,x_1)}\ldots 
 T_{(s_{n-1},x_{n-1})(s_{n},x_n)}\\
 &=\bt_{(s_0,x_0)}L_{x_0x_1}\ldots L_{x_{n-1}{x_n}}
\end{align*}
if \,$s_0 \arr{x_0}{y_0} s_1 \arr{x_1}{y_1}\ldots \arr{x_n}{y_n} s_{n+1}$\, 
is a valid path for some $y_0,y_1,\dots,y_n\in X$ and $s_{n+1}\in S$.  
Otherwise the cylinder has measure $0$. 

The measure $P$ is shift-invariant since $\bt$ is a stationary vector of 
$T$.  If the Markov chain defined by the matrix $T$ is irreducible, then 
the vector $\bt$ is unique and positive, and the measure $P$ is ergodic, 
but we do not make this assumption yet.  We do know that the vector $\bl$ 
is positive.  Hence for every finite word $\tilde w\in(S\times X)^*$,
\[
P \bigl((g, x)\tilde w(S \times X)^\nn\bigr) = \frac{\bt(g,x)}{\bl(x)}
 \pigs\mu\bigl( \tilde w(S \times X)^\nn\bigr).
\]
Therefore for each cylinder $\ogx =(g,x)(S\times X)^\nn$ we obtain
\[
P|_\ogx = \frac{\bt(g,x)}{\bl(x)} \pigs\mu|_\ogx.
\]
Since the image of the map $\tilde\pi_g$ is contained in the union of the 
cylinders $\ogx$, $x\in X$, the measure $\pigs \mu$ is supported on that 
union.  It follows that
\[
\pigs\mu = \sum_{x\in X} \frac{\bl(x)}{\bt(g,x)}P|_\ogx
\]
provided that $\bt(g,x)>0$ for all $x\in X$.  From this observation we 
derive the following lemma.

\begin{lemma}\label{pmllp}
If $\bt(g,x)>0$ for all $x\in X$, then the measure $\pigs\mu$ is absolutely 
continuous with respect to $P$.
\end{lemma}

Next we introduce the measure $Q=\lags P$.  Note that $g_*\mu= 
\lags\pigs\mu$.  Since $\tilde\la$ is a $1$-block factor map, properties of 
the measures $P$ and $\pigs\mu$ translate into analogous properties of $Q$ 
and $g_*\mu$. 

\begin{lemma}\label{PQ}
The measure $Q$ is shift-invariant.  It is ergodic whenever $P$ is ergodic.
\end{lemma}

\begin{proof}
Since $\tilde\la$ is a block factor map, it intertwines the shifts on 
$(S\times X)^\nn$ and $X^\nn$ so that we have the following commutative 
diagram:
\[
\begin{diagram}
(S\times X)^\nn & \rTo^{\sigma}&(S\times X)^\nn  \\
\dOnto^{\tilde\lambda}&  & \dOnto^{\tilde\lambda}  \\
X^\nn &  \rTo^{\sigma} & X^\nn \\
\end{diagram}
\]
By construction, the measure $P$ is shift-invariant, that is, 
$P(\si^{-1}(\widetilde E))=P(\widetilde E)$ for any measurable set 
$\widetilde E\subset(S\times X)^\nn$.  Then for any measurable set 
$E\subset X^\nn$,
\[
Q(\sigma^{-1}(E)) = P(\lag^{-1}(\sigma^{-1}(E)))
=P(\sigma^{-1}(\lag^{-1}(E)))= P(\lag^{-1}(E))= Q(E).
\]
Hence $Q$ is shift-invariant as well.

Now assume that $P$ is ergodic, that is, for any measurable set $\widetilde 
E\subset(S\times X)^\nn$ invariant under the shift, $\sigma^{-1}(\widetilde 
E)=\widetilde E$, we have $P(\widetilde E)=0$ or $1$.  Let $E$ be a 
measurable subset of $X^\nn$ invariant under the shift.  Then $\widetilde 
E=\lag^{-1}(E)$ is also invariant under the shift and $Q(E)=P(\widetilde 
E)$.  Hence $Q(E) = 0$ or $1$.  Thus the measure $Q$ is ergodic as well. 
\end{proof}

\begin{lemma}\label{gsmuQ}
If $\bt(g,x)>0$ for all $x\in X$, then the measure $g_*\mu$ is absolutely 
continuous with respect to $Q$.
\end{lemma}

\begin{proof}
We need to show that $Q(E)=0$ implies $g_*\mu(E)=0$ for any measurable set 
$E\subset X^\nn$.  Let $\widetilde E=\lag^{-1}(E)$.  Then $P(\widetilde E) 
=Q(E)=0$.  By Lemma \ref{pmllp}, the measure $\pigs\mu$ is absolutely 
continuous with respect to $P$.  Hence $\pigs\mu(\widetilde E)=0$.  Since 
$g_*\mu=\lags\pigs\mu$, it follows that $g_*\mu(E)=\pigs\mu(\widetilde 
E)=0$.
\end{proof}

Now let us discuss when the Markov chain defined by the matrix $T$ is 
irreducible.  An obvious necessary condition is that the automaton $\cA$ be 
strongly connected.  If all entries of the matrix $L$ are positive (for 
example, if the measure $\mu$ is Bernoulli), this condition is also 
sufficient.  However it need not be so for a general Markov measure.

\begin{example}
Let $X=\{0, 1, 2\}$, the measure $\mu$ be defined by the matrix
\[
L=\left(
\begin{array}{ccc}
1/2 & 1/2 & 0 \\
0 & 1/2 & 1/2 \\
1/2& 0 & 1/2
\end{array}
\right),
\]
and the automaton $\cA$ have the transition function given by the diagram 
in Figure \ref{ex1}.

\begin{figure}[h!]
\centering
\includegraphics[width=3in]{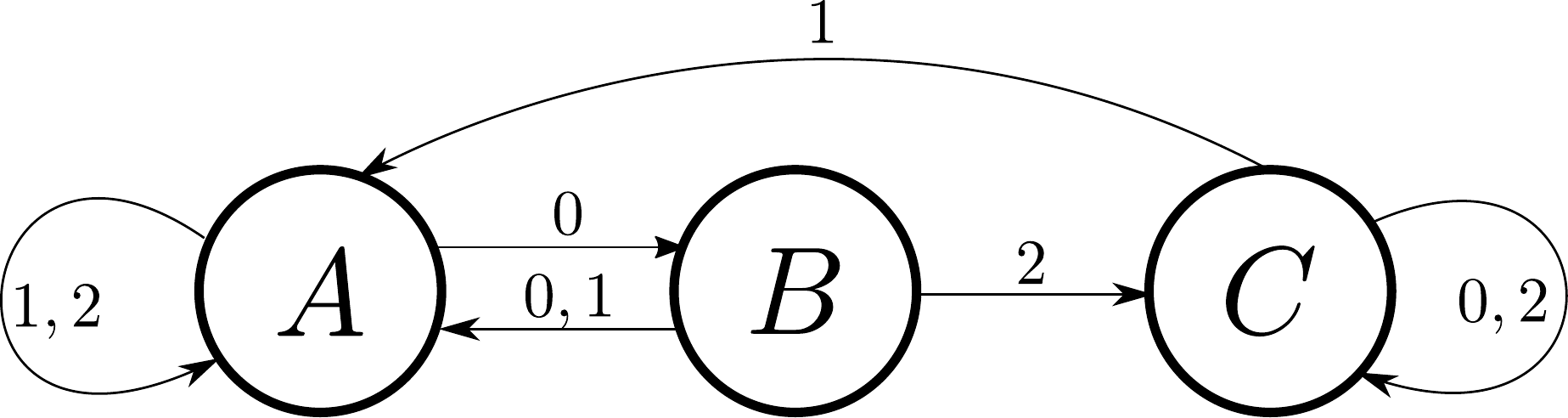}
\caption{Any path from $A$ to $C$ ends in $02$}
\label{ex1}
\end{figure}

Note that $02$ is a forbidden word in the Markov chain defined by $L$, that 
is, $\mu(wX^\nn)=0$ whenever $02$ is a subword of $w$.  On the other hand, 
any input word that takes the automaton from state $A$ to state $C$ must 
end in $02$.  Therefore in the Markov chain defined by $T$ there is zero 
chance to get from $(A,x)$ to $(C,y)$ in any number of steps.  Thus the 
Markov chain is not irreducible. \tri
\end{example}

The above example motivates the following definition.

\begin{definition}
We say that the automaton $\cA$ is \textbf{$L$-strongly connected} if for 
any pair of states $s, r\in S$ and any pair of symbols $x, y \in X$, there 
exists a word $w \in X^*$ such that $\pi(s, xw) =r$ and $xwy$ is not a 
forbidden word in the Markov chain defined by the matrix $L$ (that is, if 
$w=w_1\ldots w_n$ then $L_{xw_1}$, $L_{w_iw_{i+1}}$ for $1\le i\le n-1$, 
and $L_{w_ny}$ are all nonzero).
\end{definition}

\begin{lemma}\label{markovirr}
The Markov chain defined by the matrix $T=T_{L,\cA}$ is irreducible if and 
only if the automaton $\cA$ is $L$-strongly connected.
\end{lemma}

\begin{proof}
This follows directly from the definitions.
\end{proof}

Finally we can formulate the main results of this section.

\begin{theorem}\label{mainthm}
Let $\mu$ be a Markov measure on $X^\nn$ defined by an irreducible 
stochastic matrix $L$.  Suppose a transformation $g:X^\nn\to X^\nn$ is 
generated by an automaton $\cA$.  If the automaton $\cA$ is $L$-strongly 
connected then for any $x\in X$ and $\mu$-almost all $\om\in X^\nn$,
\[
\lim_{n\to\infty}\frac{1}{n} \sum_{i=0}^{n-1} \chi_{xX^\nn}(\sigma^ig(\om)) 
= \sum_{s_0 \arr{x_0}{x} s_1} \bt(s_0, x_0),
\]
where $\bt$ is the stationary probability vector of the stochastic matrix 
$T=T_{L,\cA}$ defined in \eqref{eq:T} and the sum is over edges in the 
Moore diagram of $\cA$.
\end{theorem}

\begin{proof}
We are going to use the measures $P$ and $Q$ defined above.  Since the 
automaton $\cA$ is $L$-strongly connected, the Markov chain defined by the
matrix $T$ is irreducible due to Lemma \ref{markovirr}.  It follows that 
the stationary vector $\bt$ is unique and positive.  Besides, the Markov 
measure $P$ defined by $T$ and $\bt$ is ergodic.  Then Lemma \ref{PQ} 
implies that the measure $Q$ is also shift-invariant and ergodic.  By the 
Birkhoff ergodic theorem, for $Q$-almost all $v\in X^\nn$ we have
\begin{align*}
\lim_{n\to\infty}\frac{1}{n} \sum_{i=0}^{n-1} \chi_{xX^\nn}(\sigma^i(v))
&= \int_{X^\nn} \chi_{xX^\nn} \,dQ
= Q(xX^\nn)
= (\lags P)(xX^\nn)\\
&= \sum_{s_0 \arr{x_0}{x} s_1} P\bigl((s_0, x_0)(S\times X)^\nn\bigr)
= \sum_{s_0 \arr{x_0}{x} s_1} \bt(s_0,x_0),
\end{align*}
where the last two sums are over valid edges in the Moore diagram of the 
automaton $\cA$.  Since all coordinates of the vector $\bt$ are positive, 
Lemma \ref{gsmuQ} implies that the measure $g_*\mu$ is absolutely 
continuous with respect to $Q$.  Therefore the above equality also holds 
for $g_*\mu$-almost all $v\in X^\nn$.  In other words, if $v=g(\om)$ then 
the equality holds for $\mu$-almost all $\om\in X^\nn$.
\end{proof}

Theorem \ref{mainthm} allows to calculate frequencies with which various 
symbols $x\in X$ appear in a sequence $g(\om)$, where $\om$ is 
$\mu$-generic.  A generalization to frequencies of arbitrary words over the 
alphabet $X$ is straightforward.

\begin{theorem}\label{mainthm1}
Let $\mu$ be a Markov measure on $X^\nn$ defined by an irreducible 
stochastic matrix $L$.  Suppose a transformation $g:X^\nn\to X^\nn$ is 
generated by an automaton $\cA$.  If the automaton $\cA$ is $L$-strongly 
connected then for any nonempty word $u=u_1u_2\dots u_k\in X^*$ and 
$\mu$-almost all $\om\in X^\nn$,
\[
\lim_{n\to\infty}\frac{1}{n}\sum_{i=0}^{n-1} \chi_{uX^\nn}(\sigma^ig(\om)) =
\sum_{s_0 \arr{x_0}{u_1} s_1 \ldots \arr{x_{k-1}}{u_k} s_k}
\bt_{(s_0, x_0)} L_{x_0x_1}\ldots L_{x_{k-2}x_{k-1}},
\]
where $\bt$ is the stationary probability vector of the stochastic matrix 
$T=T_{L,\cA}$ defined in \eqref{eq:T} and the sum is over paths in the 
Moore diagram of $\cA$.
\end{theorem}

The proof is completely analogous to that of Theorem \ref{mainthm} and we 
omit it.  For examples of calculations using Theorems \ref{mainthm} and 
\ref{mainthm1}, see Section \ref{examples} below.

Suppose $K=(K_{ss'})_{s,s'\in S}$ is a stochastic matrix that defines a 
Markov chain on the set $S$.  The tensor product $K\otimes L$ is an array 
of numbers indexed by two states $s,s'\in S$ and two symbols $x,x'\in X$, 
and given by $(K\otimes L)_{(s,x)(s',x')}=K_{ss'}L_{xx'}$.  We regard 
$K\otimes L$ as a matrix which rows and columns are indexed by elements of 
$S\times X$, not as a $4$-dimensional array.  Then $K\otimes L$ is 
stochastic and defines a Markov chain on $S\times X$.  Suppose $\bk$ is a 
stationary probability vector of the matrix $K$.  The tensor product 
$\bk\otimes\bl$ is an array of numbers indexed by pairs $(s,x)\in S\times 
X$ and given by $(\bk\otimes\bl)_{(s,x)}=\bk_s\bl_x$.  We regard it as a 
vector, not as a matrix.  Then $\bk\otimes\bl$ is a stationary probability 
vector of the matrix $K\otimes L$.

Recall that the matrix $T=T_{L,\cA}$ defines a Markov chain on $S\times 
X$.  Unfortunately, $T$ cannot be represented as the tensor product of $L$ 
with another stochastic matrix.  Nevertheless, in some cases the stationary 
probability vector $\bt$ does decompose as the tensor product of $\bl$ with 
another probability vector, which allows to simplify the formulas in 
Theorems \ref{mainthm} and \ref{mainthm1}.

Let us define a matrix $K=K_{\bl,\cA}$ by
\begin{equation}\label{eq:K}
K_{s_0s_1}= \sum_{x:\, \pi(s_0, x) = s_1} \bl(x)
=\sum_{s_0\arr{x}{y}s_1} \bl(x)
\end{equation}
for all $s_0,s_1\in S$ (in the second formula, the sum is over valid edges 
in the Moore diagram of the automaton $\cA$).  For any $s\in S$ we have 
$\sum_{r}K_{sr}=\sum_x \bl(x)=1$ so that $K$ is indeed a stochastic matrix.
Since the vector $\bl$ is positive, it follows that $K$ is irreducible if 
and only if the automaton $\cA$ is strongly connected.

\begin{lemma}\label{tensorB}
Suppose $\bk$ is a stationary probability vector of $K$.  If the Markov 
measure defined by $L$ is Bernoulli, then $\bk\otimes\bl$ is a stationary 
probability vector of $T$.
\end{lemma}

\begin{proof}
The vector $\bk=(\bk_s)_{s\in S}$ satisfies $\sum_{s}\bk_sK_{ss'}=\bk_{s'}$ 
for all $s'\in S$.  We need to show that $\sum_{s,x} \bk_s\bl_x 
T_{(s,x)(s',x')}=\bk_{s'}\bl_{x'}$ for all $(s',x')\in S\times X$.  The 
Markov measure defined by $L$ is Bernoulli if $L_{xx'}=\bl_{x'}$ for all 
$x,x'\in X$.  Then $T_{(s,x)(s',x')}=\bl_{x'}$ if $\pi(s,x)=s'$ and $0$ 
otherwise.  It follows that $\sum_{x} \bl_xT_{(s,x)(s',x')} 
=K_{ss'}\bl_{x'}$.  Consequently, $\sum_{s,x} \bk_s\bl_x 
T_{(s,x)(s',x')}=\sum_s \bk_sK_{ss'}\bl_{x'} =\bk_{s'}\bl_{x'}$.
\end{proof}

Combining Lemma \ref{tensorB} with Theorem \ref{mainthm1}, we obtain the 
following result (in the case of one-letter words, it was proved by 
Kravchenko \cite{Krav}).

\begin{theorem}\label{mainthmB}
Let $\mu$ be a Bernoulli measure on $X^\nn$ defined by a positive 
probability vector $\bl$.  Suppose a transformation $g:X^\nn\to X^\nn$ is 
generated by an automaton $\cA$.  If the automaton $\cA$ is strongly 
connected then for any nonempty word $u=u_1u_2\dots u_k\in X^*$ and 
$\mu$-almost all $\om\in X^\nn$,
\[
\lim_{n\to\infty}\frac{1}{n}\sum_{i=0}^{n-1} \chi_{uX^\nn}(\sigma^ig(\om)) =
\sum_{s_0 \arr{x_0}{u_1} s_1 \ldots \arr{x_{k-1}}{u_k} s_k}
\bk_{s_0}\bl_{x_0}\bl_{x_1}\ldots \bl_{x_{k-1}},
\]
where $\bk$ is the stationary probability vector of the stochastic matrix 
$K=K_{\bl,\cA}$ defined in \eqref{eq:K} and the sum is over paths in the 
Moore diagram of $\cA$.
\end{theorem}

If the Markov measure defined by $L$ is not Bernoulli, the vector $\bt$ 
need not decompose as $\bk\otimes\bl$ (see Example \ref{non-tens} below).  
However there is a large, well known class of automata for which Lemma 
\ref{tensorB} does hold for a general stochastic matrix $L$.  We consider 
that class in the next section.

\section{Reversible automata}\label{reversible}

\begin{definition}
An automaton $\cA=(X,S,\pi,\la)$ is called \textbf{reversible} if for any 
$s\in S$ and any $x\in X$ there exists a unique state $s_0\in S$ such that 
$\pi(s_0,x)=s$.
\end{definition}

Suppose $\cA=(X,S,\pi,\la)$ is a reversible automaton.  For any $s\in S$ 
and any $x\in X$ let $\overleftarrow{\pi}(s,x)$ be a unique state such that 
$\pi\bigl(\overleftarrow{\pi}(s,x),x\bigr)=s$.  Also, let 
$\overleftarrow{\la}(s,x)=\la\bigl(\overleftarrow{\pi}(s,x),x\bigr)$.  Then 
$\overleftarrow{\cA}=(S,X,\overleftarrow{\pi},\overleftarrow{\la})$ is 
called the \textbf{reverse automaton} of $\cA$.  In terms of the Moore 
diagrams, the automaton $\overleftarrow{\cA}$ is obtained from $\cA$ by 
reversing all edges.  That is, every edge of the form $s_0\arr{x}{y}s_1$ is 
replaced by $s_1\arr{x}{y}s_0$.  The automaton $\overleftarrow{\cA}$ is 
also reversible and its reverse automaton is $\cA$.

Suppose $L$ is a stochastic matrix that defines a Markov chain on $X$ and 
$\bl$ is a stationary probability vector of $L$. 

\begin{lemma}\label{tensor-rev}
If an automaton $\cA=(X,S,\pi,\la)$ is reversible then the constant vector 
$\bk=\frac{1}{|S|}(1,1,\ldots,1)$ is a stationary probability vector of the 
stochastic matrix $K=K_{\bl,\cA}$ defined in \eqref{eq:K} while $\bk\otimes 
\bl$ is a stationary probability vector of the stochastic matrix 
$T=T_{L,\cA}$ defined in \eqref{eq:T}.
\end{lemma}

\begin{proof}
For any $s,s'\in S$,
\[
K_{\bl,\cA}(s,s')=\sum_{s\arr{x}{y}s'} \bl(x).
\]
It follows that the transpose of the matrix $K_{\bl,\cA}$ is 
$K_{\bl,\overleftarrow{\cA}}$.  As a consequence, the transpose is 
stochastic as well.  Then $\sum_s K_{\bl,\cA}(s,s')=1$ for all $s'\in S$, 
which implies that $\bk K_{\bl,\cA}=\bk$.

To prove the second statement of the lemma, it is enough to show that
\[
\sum_{s,x} \bl_xT_{(s,x)(s',x')}=\bl_{x'}
\]
for all $(s',x')\in S\times X$.  Note that $T_{(s,x)(s',x')}=L_{xx'}$ if 
$s=\overleftarrow{\pi}(s',x)$ and $0$ otherwise.  It follows that 
$\sum_{s}T_{(s,x)(s',x')}=L_{xx'}$.  Then $\sum_{s,x}\bl_xT_{(s,x)(s',x')}= 
\sum_x \bl_xL_{xx'}=\bl_{x'}$.
\end{proof}

Combining Lemma \ref{tensor-rev} with Theorem \ref{mainthm1}, we obtain the 
following result.

\begin{theorem}\label{mainthm-rev}
Let $\mu$ be a Markov measure on $X^\nn$ defined by an irreducible 
stochastic matrix $L$ with stationary probability vector $\bl$.  Suppose a 
transformation $g:X^\nn\to X^\nn$ is generated by an automaton $\cA$.  If 
the automaton $\cA$ is $L$-strongly connected and reversible, then for any 
nonempty word $u=u_1u_2\dots u_k\in X^*$ and $\mu$-almost all $\om\in 
X^\nn$,
\[
\lim_{n\to\infty}\frac{1}{n}\sum_{i=0}^{n-1} \chi_{uX^\nn}(\sigma^ig(\om)) =
\frac1{N}\sum_{s_0 \arr{x_0}{u_1} s_1 \ldots \arr{x_{k-1}}{u_k} s_k}
\bl_{x_0}L_{x_0x_1}\ldots L_{x_{k-2}x_{k-1}},
\]
where $N$ is the number of states in $\cA$ and the sum is over paths in the 
Moore diagram of $\cA$.
\end{theorem}

A remarkable feature of the reversible automata is that their states act 
naturally on bi-infinite sequences over the alphabet.  Suppose 
$\cA=(X,S,\pi,\la)$ is a reversible automaton and let $w=\ldots 
x_{-2}x_{-1}x_0.x_1x_2x_3\ldots$ be a bi-infinite sequence in $X^\zz$ (the
dot between $x_0$ and $x_1$ serves as a reference point).  Given a state 
$g\in S$, we need to find a bi-infinite path in the Moore diagram of $\cA$ 
of the form
\[
\ldots \arr{x_{-2}}{y_{-2}} s_{-2}\arr{x_{-1}}{y_{-1}} s_{-1}\arr{x_0}{y_0} 
g \arr{x_1}{y_1} s_1 \arr{x_2}{y_2} s_2 \arr{x_3}{y_3} \ldots
\]
Since the automaton $\cA$ is reversible, such a path exists and is unique.  
Then, by definition, $g(w)=\ldots y_{-2}y_{-1}y_0.y_1y_2y_3\ldots$.  Let 
$g^+$ denote the action of the same state $g$ on $X^\nn$ and $g^-$ denote 
the action of $g$ on $X^\nn$ when $g$ is regarded as a state of the reverse 
automaton $\overleftarrow{\cA}$.  Then $y_1y_2y_3\ldots= 
g^+(x_1x_2x_3\ldots)$ and $y_0y_{-1}y_{-2}\ldots= 
g^-(x_0x_{-1}x_{-2}\ldots)$.

The \textbf{two-sided shift} on $X^\zz$ (still denoted by $\si$) is defined 
by
\[
\si(\ldots x_{-2}x_{-1}x_0.x_1x_2x_3\ldots)=\ldots 
x_{-1}x_0x_1.x_2x_3x_4\ldots
\]
Unlike the shift on $X^\nn$, it is invertible.  Given a stochastic matrix 
$L=(L_{xx'})_{x,x'\in X}$ with a stationary probability vector 
$\bl=(\bl_x)_{x\in X}$, a Markov measure $\mu$ on $X^\zz$ is defined on the
cylinders by
\[
\mu\bigl(\{\ldots w_{-2}w_{-1}w_0.w_1w_2\ldots\mid
w_i=x_i,\ m\le i\le n\}\bigr)=\bl_{x_m}L_{x_mx_{m+1}}\ldots L_{x_{n-1}x_n}
\]
for any $m,n\in\zz$, $m\le n$ and any $x_m,x_{m+1},\dots,x_n\in X$.  The 
measure $\mu$ is shift-invariant.  It is ergodic if $L$ is irreducible.

For any nonempty word $u=u_1u_2\dots u_k\in X^*$ consider a cylinder 
$\Omega_u\subset X^\zz$ defined by $\Omega_u=\{\ldots 
w_{-2}w_{-1}w_0.w_1w_2 \ldots \mid w_i=u_i,\ 1\le i\le k\}$.  Given $w\in 
X^\zz$, the limit
\[
\lim_{n\to\infty}\frac{1}{n}\sum_{i=0}^{n-1} \chi_{\Omega_u}(\sigma^i(w)),
\]
if it exists, yields the (asymptotic) frequency at which the word $u$ 
occurs in the right-hand half of the bi-infinite sequence $w$.  Likewise, 
the limit
\[
\lim_{n\to\infty}\frac{1}{n}\sum_{i=0}^{n-1} 
\chi_{\Omega_u}(\sigma^{-i}(w)),
\]
if it exists, yields the frequency at which $u$ occurs in the left-hand 
half of $w$.

Now we can formulate an analogue of Theorem \ref{mainthm-rev} for 
bi-infinite sequences.

\begin{theorem}\label{mainthm-bi}
Let $\mu$ be a Markov measure on $X^\zz$ defined by an irreducible 
stochastic matrix $L$ with stationary probability vector $\bl$.  Suppose a 
transformation $g:X^\zz\to X^\zz$ is generated by a reversible automaton 
$\cA$.  If the automaton $\cA$ is $L$-strongly connected then for any 
nonempty word $u=u_1u_2\dots u_k\in X^*$ and $\mu$-almost all $w\in X^\zz$,
\begin{align*}
\lim_{n\to\infty}\frac{1}{n}\sum_{i=0}^{n-1} \chi_{\Omega_u}(\sigma^ig(w)) &
=\lim_{n\to\infty}\frac{1}{n}\sum_{i=0}^{n-1} 
\chi_{\Omega_u}(\sigma^{-i}g(w))\\
& =\frac1{N}\sum_{s_0 \arr{x_0}{u_1} s_1 \ldots \arr{x_{k-1}}{u_k} s_k}
\bl_{x_0}L_{x_0x_1}\ldots L_{x_{k-2}x_{k-1}},
\end{align*}
where $N$ is the number of states in $\cA$ and the sum is over paths in the 
Moore diagram of $\cA$.
\end{theorem}

\begin{proof}
We are going to use the transformations $g^+$ and $g^-$ defined above.  Let 
us also define maps $F^+,F^-:X^\zz\to X^\nn$ by
\begin{align*}
F^+(\ldots w_{-2}w_{-1}w_0.w_1w_2w_3\ldots) &=w_1w_2w_3\ldots,\\
F^-(\ldots w_{-2}w_{-1}w_0.w_1w_2w_3\ldots) &=w_0w_{-1}w_{-2}\ldots
\end{align*}
The maps $F^+$ and $F^-$ are continuous.  Consider the pushforward measures 
$\mu^+=F^+_*\mu$ and $\mu^-=F^-_*\mu$ on $X^\nn$.  The measure $\mu^+$ is 
clearly the Markov measure on $X^\nn$ defined by the same matrix $L$ and 
vector $\bl$.  By Theorem \ref{mainthm-rev},
\[
\lim_{n\to\infty}\frac{1}{n}\sum_{i=0}^{n-1} 
\chi_{uX^\nn}(\sigma^ig^+(\om)) =
\frac1{N}\sum_{s_0 \arr{x_0}{u_1} s_1 \ldots \arr{x_{k-1}}{u_k} s_k}
\bl_{x_0}L_{x_0x_1}\ldots L_{x_{k-2}x_{k-1}}
\]
for $\mu^+$-almost all $\om\in X^\nn$.  In other words, if $\om=F^+(w)$ 
then the latter equality holds for $\mu$-almost all $w\in X^\zz$.  Since
\[
\chi_{uX^\nn}(\sigma^ig^+(F^+(w)))=\chi_{uX^\nn}(\sigma^iF^+(g(w)))
=\chi_{uX^\nn}(F^+(\sigma^ig(w)))=\chi_{\Omega_u}(\sigma^ig(w))
\]
for all $w\in X^\zz$ and $i\ge0$, this establishes the first limit in the 
formulation of the theorem.

The second limit requires more work.  The measure $\mu^-$ is given on the 
cylinders by
\[
\mu^-(y_1y_2\ldots y_mX^\nn)=\bl_{y_m}L_{y_my_{m-1}}\ldots L_{y_2y_1}.
\]
Consider a matrix $\overleftarrow{L}=(\overleftarrow{L}_{xx'})_{x,x'\in X}$
defined by $\overleftarrow{L}_{xx'}=\bl_{x'}L_{x'x}/\bl_x$ for all $x,x'\in 
X$ (note that the vector $\bl$ is positive).  The matrix 
$\overleftarrow{L}$ is stochastic.  Indeed, 
$\sum_{x'}\overleftarrow{L}_{xx'}= \sum_{x'}\bl_{x'}L_{x'x}/\bl_x 
=\bl_x/\bl_x=1$ for all $x\in X$.  Also, $\bl$ is a stationary probability 
vector of $\overleftarrow{L}$ since $\sum_x \bl_x\overleftarrow{L}_{xx'}= 
\sum_x \bl_{x'}L_{x'x}=\bl_{x'}$ for all $x'\in X$.  Now for any 
$y_1,y_2,\dots,y_m\in X$ we obtain
\begin{align*}
\bl_{y_1}\overleftarrow{L}_{y_1y_2}\overleftarrow{L}_{y_2y_3}\ldots 
\overleftarrow{L}_{y_{m-1}y_m}
&=\bl_{y_1}(\bl_{y_2}L_{y_2y_1}/\bl_{y_1})(\bl_{y_3}L_{y_3y_2}/\bl_{y_2})
\ldots (\bl_{y_m}L_{y_my_{m-1}}/\bl_{y_{m-1}})\\
&=\bl_{y_m}L_{y_my_{m-1}}\ldots L_{y_3y_2}L_{y_2y_1},
\end{align*}
which implies that $\mu^-$ is the Markov measure defined by the matrix 
$\overleftarrow{L}$ with stationary probability vector $\bl$.

By construction, $\overleftarrow{L}_{xx'}>0$ if and only if $L_{x'x}>0$.  
Since the stochastic matrix $L$ is irreducible, it follows that 
$\overleftarrow{L}$ is irreducible as well.  Next let us show that the 
reverse automaton $\overleftarrow{\cA}$ is $\overleftarrow{L}$-strongly 
connected.  Given states $s,s'\in S$ and symbols $x,x'\in X$, we need to 
find symbols $x_0=x,x_1,\dots,x_m=x'$ ($m\ge1$) such that 
$\overleftarrow{\pi}(s,x_0x_1\ldots x_{m-1})=s'$ and 
$\overleftarrow{L}_{x_ix_{i+1}}>0$ for $0\le i\le m-1$.  Let 
$r=\overleftarrow{\pi}(s,x)$ and $r'=\overleftarrow{\pi}(s',x')$.  Since 
the automaton $\cA$ is $L$-strongly connected, there exist symbols 
$y_0=x',y_1,\dots,y_j=x$ ($j\ge1$) such that $\pi(r',y_0y_1\ldots 
y_{j-1})=r$ and $L_{y_iy_{i+1}}>0$ for $0\le i\le j-1$.  Then 
$s=\pi(s',y_1y_2\ldots y_j)$ so that $s'=\overleftarrow{\pi}(s,y_j\ldots 
y_2y_1)$.  Moreover, $\overleftarrow{L}_{y_iy_{i-1}}>0$ for $1\le i\le j$.

Applying Theorem \ref{mainthm-rev} to the measure $\mu^-$, the matrix 
$\overleftarrow{L}$, the transformation $g^-$, the automaton 
$\overleftarrow{\cA}$ and the word $\overleftarrow{u}=u_ku_{k-1}\ldots u_1$ 
(which is $u$ written backwards), we obtain that for $\mu^-$-almost all 
$\om\in X^\nn$,
\[
\lim_{n\to\infty}\frac{1}{n}\sum_{i=0}^{n-1} 
\chi_{\overleftarrow{u}X^\nn}(\sigma^ig^-(\om)) =
\frac1{N}\sum_{\overleftarrow{\cA}:\,s_0 \arr{x_0}{u_k} s_1 \ldots 
\arr{x_{k-1}}{u_1} s_k} \bl_{x_0}\overleftarrow{L}_{x_0x_1}\ldots 
\overleftarrow{L}_{x_{k-2}x_{k-1}},
\]
where the sum is over paths in the Moore diagram of $\overleftarrow{\cA}$.  
In other words, if $\om=F^-(w)$ then the latter equality holds for 
$\mu$-almost all $w\in X^\zz$.  By construction of the automaton 
$\overleftarrow{\cA}$, its Moore diagram admits a path \,$s_0 
\arr{x_0}{u_k} s_1 \ldots \arr{x_{k-1}}{u_1} s_k$\, if and only if the 
Moore diagram of $\cA$ admits the path \,$s_k \arr{x_{k-1}}{u_1} \ldots s_1 
\arr{x_0}{u_k} s_0$.  By the above, $\bl_{x_0}\overleftarrow{L}_{x_0x_1} 
\ldots \overleftarrow{L}_{x_{k-2}x_{k-1}}=\bl_{x_{k-1}}L_{x_{k-1}x_{k-2}} 
\ldots L_{x_1x_0}$ for all $x_0,x_1,\dots,x_{k-1}\in X$.  It follows that 
the right-hand side in the last formula (that is, the value of the limit) 
is the same as in the formulation of the theorem.  As for the left-hand 
side, we have $\sigma^ig^-(F^-(w))=\sigma^iF^-(g(w))=F^-(\sigma^{-i}g(w))$ 
for all $w\in X^\zz$ and $i\ge0$.  Besides, 
$\chi_{\Omega_u}(\si^{-i}(\tilde w))= 
\chi_{\overleftarrow{u}X^\nn}(F^-(\si^{-i+k}(\tilde w)))$ for all $\tilde 
w\in X^\zz$ and $i\in\zz$ (here $k$ is the length of the word $u$).  It 
follows that
\[
\lim_{n\to\infty}\frac{1}{n}\sum_{i=0}^{n-1} 
\chi_{\Omega_u}(\sigma^{-i}g(w))
=\lim_{n\to\infty}\frac{1}{n}\sum_{i=0}^{n-1} 
\chi_{\overleftarrow{u}X^\nn}(\sigma^ig^-(F^-(w)))
\]
whenever the latter limit exists.  This completes the proof.
\end{proof}

\newpage
\section{Examples with strongly connected automata}\label{examples}

In this section we consider several examples of automaton transformations 
generated by strongly connected automata and perform for them calculations 
related to results of Sections \ref{strong} and \ref{reversible}.

In each example, an automaton $\cA=(X,S,\pi,\la)$ is given by its Moore 
diagram.  A shift-invariant, ergodic Markov measure $\mu$ on $X^\nn$ is 
defined by an irreducible stochastic matrix $L$ with stationary probability 
vector $\bl$ ($\bl L=\bl$ and $\sum_x \bl_x=1$).  In most examples, the 
alphabet is $X=\{0,1\}$ and the matrix $L$ is in general form
\[
L=\(
\begin{array}{cc}
1-p & p \\
q & 1-q \\
\end{array}
\),
\]
that is, $p>0$ is the probability of transition from $0$ to $1$ and $q>0$ 
is the probability of transition from $1$ to $0$.  Then
\[
\bl=\left(\frac{q}{p+q},\,\frac{p}{p+q}\right).
\]
In all examples, we compute the matrix $T=T_{L,\cA}$ defined in 
\eqref{eq:T} and find its stationary probability vector $\bt$.  Rows and 
columns of $T$ as well as coordinates of $\bt$ are indexed by elements of 
$S\times X$.  The sets $S$ and $X$ are canonically ordered as their 
elements are either letters or digits.  We impose the lexicographic order 
on the set $S\times X$, which allows us to write $T$ as a usual matrix and 
$\bt$ as a usual row vector.

In addition, we compute the matrix $K$ defined in \eqref{eq:K} and find its 
stationary probability vector $\bk$ to check if the vector $\bt$ decomposes 
as the tensor product $\bk\otimes\bl$.

Finally, we calculate the vector $\bbf=(\bbf_x)_{x\in X}$ of frequencies of 
each character $x\in X$ after the action of $\cA$ with an initial state $g$ 
on a $\mu$-generic sequence:
\[
\bbf_x = \lim_{n\to\infty} \frac{1}{n} \sum_{i=0}^{n-1} \chi_{xX^\nn} 
(\sigma^i g(\om))
\]
for $\mu$-almost all $\om\in X^\nn$, where $\sigma$ denotes the shift.  If 
the automaton is $L$-strongly connected, then the vector $\bbf$ of output 
frequencies does not depend on the initial state $g$ (as easily follows 
from Theorem \ref{mainthm}).

\newpage
\subsection{Automaton generating a free group}\label{Alesh}

The three states of this automaton generate a free nonabelian group.  
Moreover, this is essentially the only $3$-state automaton over the 
alphabet $X=\{0,1\}$ with that property (see \cite{auto-class}).  What is 
more important for us is that the automaton is reversible (in fact, 
bireversible: its inverse is reversible as well), and hence 
$\bt=\bk\otimes\bl$.

\centerline{
\includegraphics[trim=0 0.4in 0 0.35in,scale=1]{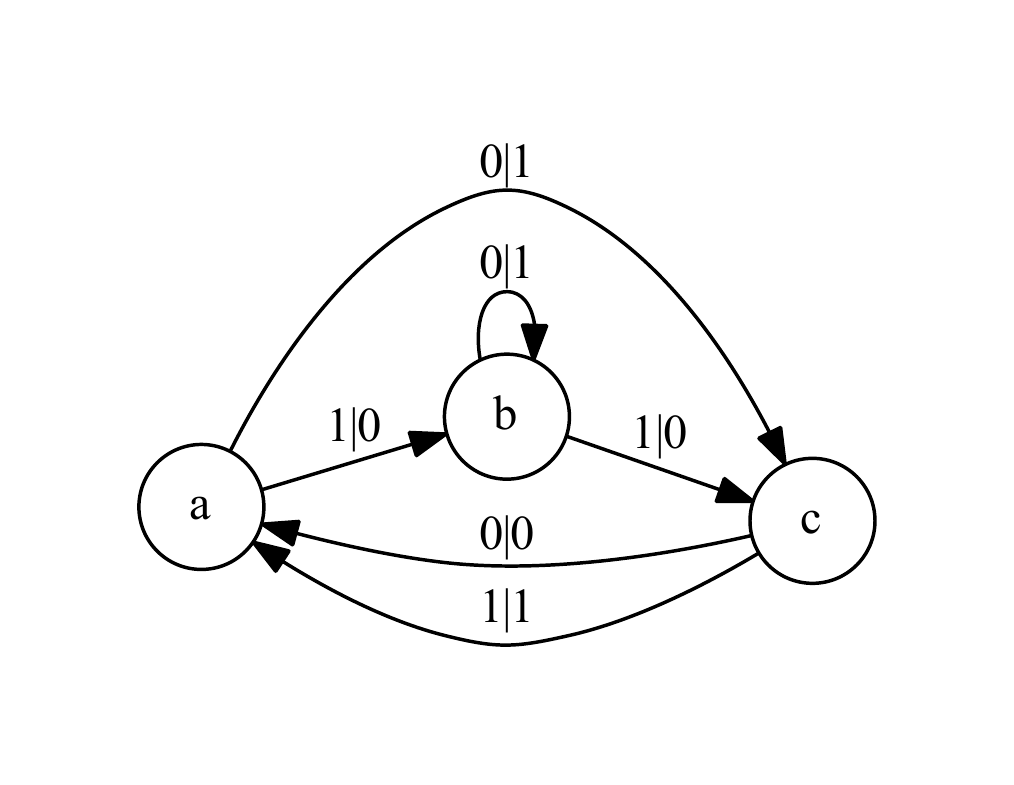}
}

\begin{align*}
 K&=\left(
\begin{array}{ccc}
 0 & \frac{p}{p+q} & \frac{q}{p+q} \\
 0 & \frac{q}{p+q} & \frac{p}{p+q} \\
 1 & 0 & 0 \\
\end{array}
\right) \\
 \bk&=\left(
\begin{array}{ccc}
 \frac{1}{3} & \frac{1}{3} & \frac{1}{3} \\
\end{array}
\right) \\
 T&=\left(
\begin{array}{cccccc}
 0 & 0 & 0 & 0 & 1-p & p \\
 0 & 0 & q & 1-q & 0 & 0 \\
 0 & 0 & 1-p & p & 0 & 0 \\
 0 & 0 & 0 & 0 & q & 1-q \\
 1-p & p & 0 & 0 & 0 & 0 \\
 q & 1-q & 0 & 0 & 0 & 0 \\
\end{array}
\right) \\
 \bt&=\left(
\begin{array}{cccccc}
 \frac{q}{3 (p+q)} & \frac{p}{3 (p+q)} & \frac{q}{3 (p+q)} & \frac{p}{3 (p+q)} & \frac{q}{3 (p+q)} & \frac{p}{3 (p+q)} \\
\end{array}
\right) \\
 \bk\otimes\bl&=\left(
\begin{array}{cccccc}
 \frac{q}{3 (p+q)} & \frac{p}{3 (p+q)} & \frac{q}{3 (p+q)} & \frac{p}{3 (p+q)} & \frac{q}{3 (p+q)} & \frac{p}{3 (p+q)} \\
\end{array}
\right) \\
 \bbf&=\left(
\begin{array}{cc}
 \frac{2 p+q}{3 (p+q)} & \frac{p+2 q}{3 (p+q)} \\
\end{array}
\right)
\end{align*}

%\newpage
\subsection{The Bellaterra automaton}

The Bellaterra automaton is obtained by composing the automaton from the 
previous example with a one-state automaton that switches $0$ and $1$.  In 
other words, all values of the transition function are retained while all 
values of the output function are switched.  This significantly changes the 
character of transformations generated by the automaton: they are all 
involutions now (see \cite{auto-class}).  However the action on Markov 
measures is not that much different: the coordinates of the vector $\bbf$ 
are interchanged while the other data remain the same.

\centerline{
\includegraphics[trim=0 0.4in 0 
0.4in,scale=1]{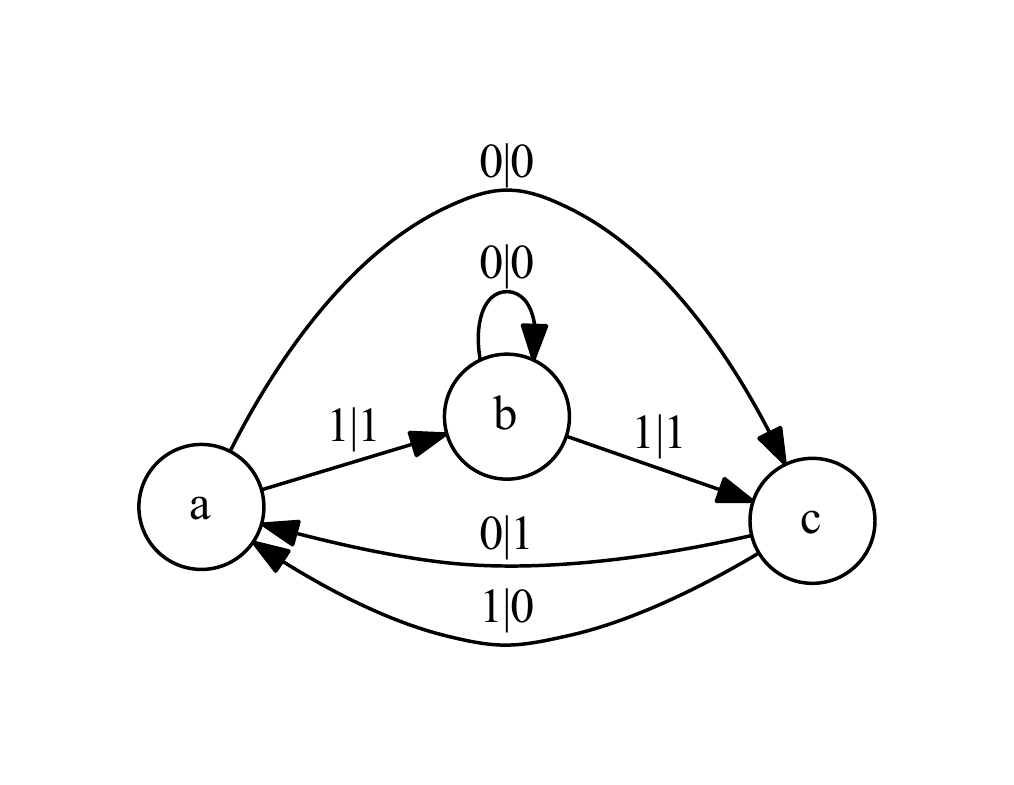}
}

\begin{align*}
 K&=\left(
\begin{array}{ccc}
 0 & \frac{p}{p+q} & \frac{q}{p+q} \\
 0 & \frac{q}{p+q} & \frac{p}{p+q} \\
 1 & 0 & 0 \\
\end{array}
\right) \\
 \bk&=\left(
\begin{array}{ccc}
 \frac{1}{3} & \frac{1}{3} & \frac{1}{3} \\
\end{array}
\right) \\
 T&=\left(
\begin{array}{cccccc}
 0 & 0 & 0 & 0 & 1-p & p \\
 0 & 0 & q & 1-q & 0 & 0 \\
 0 & 0 & 1-p & p & 0 & 0 \\
 0 & 0 & 0 & 0 & q & 1-q \\
 1-p & p & 0 & 0 & 0 & 0 \\
 q & 1-q & 0 & 0 & 0 & 0 \\
\end{array}
\right) \\
 \bt&=\left(
\begin{array}{cccccc}
 \frac{q}{3 (p+q)} & \frac{p}{3 (p+q)} & \frac{q}{3 (p+q)} & \frac{p}{3 
 (p+q)} & \frac{q}{3 (p+q)} & \frac{p}{3 (p+q)} \\
\end{array}
\right) \\
 \bk\otimes\bl&=\left(
\begin{array}{cccccc}
 \frac{q}{3 (p+q)} & \frac{p}{3 (p+q)} & \frac{q}{3 (p+q)} & \frac{p}{3 
 (p+q)} & \frac{q}{3 (p+q)} & \frac{p}{3 (p+q)} \\
\end{array}
\right) \\
 \bbf&=\left(
\begin{array}{cc}
 \frac{p+2 q}{3 (p+q)} & \frac{2 p+q}{3 (p+q)} \\
\end{array}
\right)
\end{align*}

%\newpage
\subsection{The lamplighter automaton}

The two states of this automaton generate a group isomorphic to the 
lamplighter group $(\zz/2\zz)\wr\zz$ (see \cite{auto}).  This is again a 
reversible automaton.  An interesting feature of the lamplighter automaton 
is that the output frequencies of individual characters do not depend on 
the input frequencies.

\centerline{
\includegraphics[trim=0 0.4in 0 
0.4in,scale=1.2]{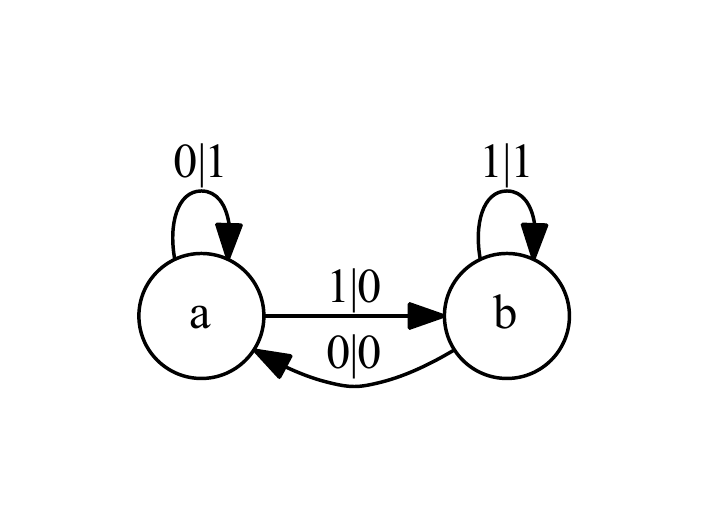}
}

\begin{align*}
 K&=\left(
\begin{array}{cc}
 \frac{q}{p+q} & \frac{p}{p+q} \\
 \frac{p}{p+q} & \frac{q}{p+q} \\
\end{array}
\right) \\
 \bk&=\left(
\begin{array}{cc}
 \frac{1}{2} & \frac{1}{2} \\
\end{array}
\right) \\
 T&=\left(
\begin{array}{cccc}
 1-p & p & 0 & 0 \\
 0 & 0 & q & 1-q \\
 0 & 0 & 1-p & p \\
 q & 1-q & 0 & 0 \\
\end{array}
\right) \\
 \bt&=\left(
\begin{array}{cccc}
 \frac{q}{2 (p+q)} & \frac{p}{2 (p+q)} & \frac{q}{2 (p+q)} & \frac{p}{2 (p+q)} \\
\end{array}
\right) \\
 \bk\otimes\bl&=\left(
\begin{array}{cccc}
 \frac{q}{2 (p+q)} & \frac{p}{2 (p+q)} & \frac{q}{2 (p+q)} & \frac{p}{2 (p+q)} \\
\end{array}
\right) \\
 \bbf&=\left(
\begin{array}{cc}
 \frac{1}{2} & \frac{1}{2} \\
\end{array}
\right) \\
\end{align*}

Note: even though the output frequencies of individual characters do not 
depend on $p$ and $q$, this is not the case for words of length $2$.  The 
input and output frequencies are as follows.

\smallskip

\[
\begin{array}{ccccc}
\mbox{Words:}& 00& 01& 10& 11\\[1mm]
\mbox{Input frequency: }& \frac{q-p q}{p+q} & \frac{p q}{p+q} & \frac{p 
q}{p+q} & \frac{p-p q}{p+q} \\[1.5mm]
\mbox{Output frequency: }& \frac{q}{2 (p+q)} & \frac{p}{2 (p+q)} & 
\frac{p}{2 (p+q)} & \frac{q}{2 (p+q)} \\
\end{array}
\]

%newpage
\subsection{Case when $\bt\neq \bk\otimes\bl$}\label{non-tens}

This can already happen with a two-character alphabet.  The automaton in 
this example differs from the automaton in Example \ref{Alesh} only by one 
arrow (in the Moore diagram), but that change makes the automaton 
non-reversible.

\centerline{
\includegraphics[trim=0 0.4in 0 0.4in,scale=1]{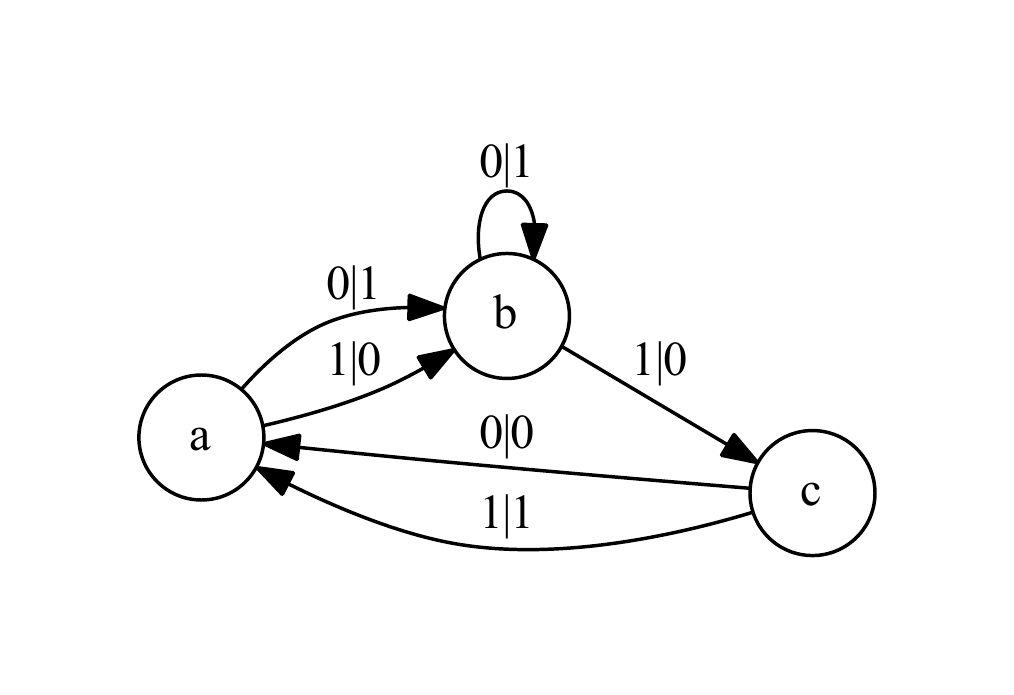}
}

\begin{align*}
 K&=\left(
\begin{array}{ccc}
 0 & 1 & 0 \\
 0 & \frac{q}{p+q} & \frac{p}{p+q} \\
 1 & 0 & 0 \\
\end{array}
\right) \\
 \bk&=\left(
\begin{array}{ccc}
 \frac{p}{3 p+q} & \frac{p+q}{3 p+q} & \frac{p}{3 p+q} \\
\end{array}
\right) \\
 T&=\left(
\begin{array}{cccccc}
 0 & 0 & 1-p & p & 0 & 0 \\
 0 & 0 & q & 1-q & 0 & 0 \\
 0 & 0 & 1-p & p & 0 & 0 \\
 0 & 0 & 0 & 0 & q & 1-q \\
 1-p & p & 0 & 0 & 0 & 0 \\
 q & 1-q & 0 & 0 & 0 & 0 \\
\end{array}
\right) \\
 \bt&=\left(
\begin{array}{ccc}
 -\frac{p q (p+q-2)}{q p^2+\left(2 q^2-3 q+3\right) p+\left(q^2-3 q+3\right) q} &
 \frac{p \left(q^2+(p-2) q+1\right)}{q p^2+\left(2 q^2-3 q+3\right) p+\left(q^2-3 q+3\right) q} &
  \frac{q \left(p^2+(2 q-3) p+q^2-3 q+3\right)}{q p^2+\left(2 q^2-3 q+3\right) p+\left(q^2-3 q+3\right) q}
  \end{array} \right.\\
  &\left.\begin{array}{ccc}
   \frac{p}{q p^2+\left(2 q^2-3 q+3\right) p+\left(q^2-3 q+3\right) q} & \frac{p q}{q p^2+\left(2 q^2-3 q+3\right) p+\left(q^2-3 q+3\right) q} & \frac{p-p q}{q p^2+\left(2 q^2-3 q+3\right) p+\left(q^2-3 q+3\right) q} \\
\end{array}
\right) \\
 \bk\otimes\bl&=\left(
\begin{array}{cccccc}
 \frac{p q}{(p+q) (3 p+q)} & \frac{p^2}{(p+q) (3 p+q)} & \frac{q}{3 p+q} & \frac{p}{3 p+q} & \frac{p q}{(p+q) (3 p+q)} & \frac{p^2}{(p+q) (3 p+q)} \\
\end{array}
\right) \\
 \bbf&=\left(
\begin{array}{cc}
 \frac{p \left(q^2+(p-1) q+2\right)}{q p^2+\left(2 q^2-3 q+3\right) p+\left(q^2-3 q+3\right) q} & \frac{p (q-1)^2+\left(q^2-3 q+3\right) q}{q p^2+\left(2 q^2-3 q+3\right) p+\left(q^2-3 q+3\right) q} \\
\end{array}
\right) \\
\end{align*}

%\newpage
\subsection{Automaton over a three-character alphabet}

In our final example, we consider an automaton $\cA$ over a three-character 
alphabet $X=\{1,2,3\}$.

\centerline{
\includegraphics[trim=0.6in 0 0 0,scale=0.85]{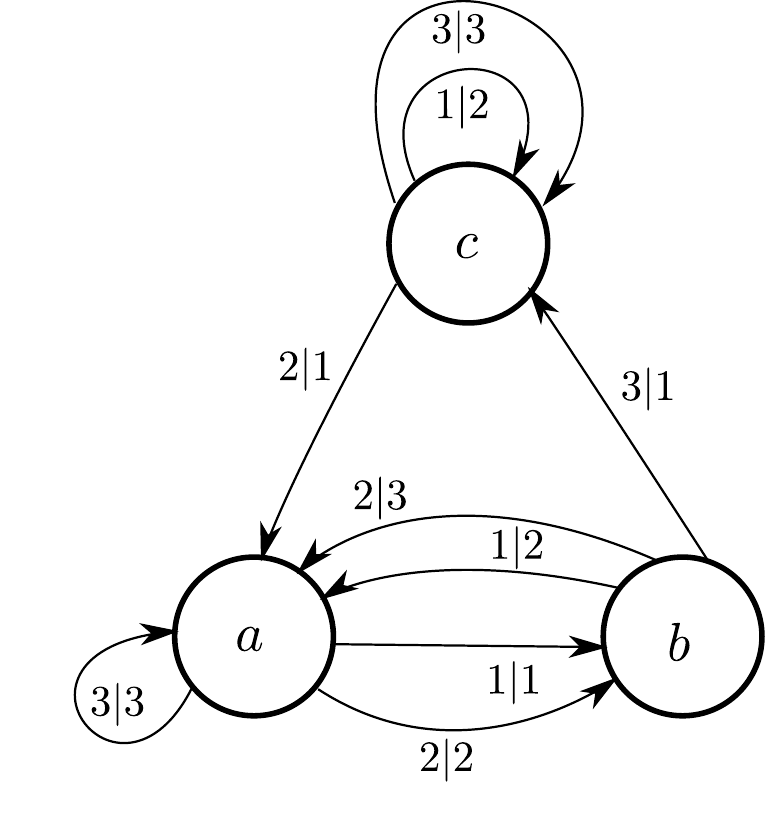}
}

Let $\mu$ be the Markov measure on $X^\nn$ defined by the matrix
\[
L=\left(
\begin{array}{ccc}
 \frac{1}{2} & \frac{1}{2} & 0 \\
 0 & \frac{1}{2} & \frac{1}{2} \\
 \frac{1}{2} & 0 & \frac{1}{2} \\
\end{array}
\right)\
\]
with stationary probability vector $\bl=(1/3,1/3,1/3)$.  Then $\cA$ is 
$L$-strongly connected and we have
\begin{align*}
T&=\left(
\begin{array}{ccccccccc}
 0 & 0 & 0 & \frac{1}{2} & \frac{1}{2} & 0 & 0 & 0 & 0 \\
 0 & 0 & 0 & 0 & \frac{1}{2} & \frac{1}{2} & 0 & 0 & 0 \\
 \frac{1}{2} & 0 & \frac{1}{2} & 0 & 0 & 0 & 0 & 0 & 0 \\
 \frac{1}{2} & \frac{1}{2} & 0 & 0 & 0 & 0 & 0 & 0 & 0 \\
 0 & \frac{1}{2} & \frac{1}{2} & 0 & 0 & 0 & 0 & 0 & 0 \\
 0 & 0 & 0 & 0 & 0 & 0 & \frac{1}{2} & 0 & \frac{1}{2} \\
 0 & 0 & 0 & 0 & 0 & 0 & \frac{1}{2} & \frac{1}{2} & 0 \\
 0 & \frac{1}{2} & \frac{1}{2} & 0 & 0 & 0 & 0 & 0 & 0 \\
 0 & 0 & 0 & 0 & 0 & 0 & \frac{1}{2} & 0 & \frac{1}{2} \\
\end{array}
\right)\\
\bt&=\left(
\begin{array}{ccccccccc}
 \frac{2}{15} & \frac{2}{15} & \frac{1}{5} & \frac{1}{15} & \frac{2}{15} & 
 \frac{1}{15} & \frac{2}{15} & \frac{1}{15} & \frac{1}{15} \\
\end{array}
\right)\\
\bbf &=
\left(
\begin{array}{ccc}
 \frac{4}{15} & \frac{1}{3} & \frac{2}{5}\\
\end{array}
\right)
\end{align*}
We did not include the matrix $K$ and the vector $\bk$ in this example, but 
it is easy to see that $\bt\neq \boldsymbol{v} \otimes\bl$ for any vector 
$\boldsymbol{v}$.

Note that if we modified the automaton $\cA$ by changing one arrow in the 
Moore diagram so that $\pi(a,2) = a$ (instead of $b$), then $\cA$ would no 
longer be $L$-strongly-connected.  Indeed, since $L_{1,3}=\mu(13X^\nn)=0$, 
there would be zero chance to get from $(a,1)$ and $(b,3)$ in the Markov 
chain defined by $T$.  As a result, the vector $\bt$ is no longer positive:
\[
\bt=
\left(
\begin{array}{ccccccccc}
 \frac{2}{9} & \frac{2}{9} & \frac{1}{3} & \frac{1}{9} & \frac{1}{9} & 0 & 
 0 & 0 & 0 \\
\end{array}
\right).
\]
This modification shows that in the case of Markov measures, $\bt$ may not 
be uniquely determined by the automaton and the vector $\bl$, like in the 
case of Bernoulli measures.  Indeed, for a Bernoulli measure with the same 
probability distribution $\bl$, the vector $\bt$ would be positive.

\section{Singularity}\label{singular}

Suppose $\mu$ is a Markov measure on $X^\nn$ and $g$ is an automaton 
transformation of $X^\nn$.  If $g$ has polynomial activity growth, the 
results of Section \ref{poly} suggest that we should expect the pushforward 
measure $g_*\mu$ to be absolutely continuous with respect to $\mu$.  In 
this section we study the relation between the measures $\mu$ and $g_*\mu$ 
in the case when $g$ is generated by a strongly connected automaton.  The 
relation turns out to be quite different, namely, we should expect $\mu$ 
and $g_*\mu$ to be singular (that is, concentrated on disjoint sets).

Kravchenko observed in \cite{Krav} that if $\mu$ is a Bernoulli measure and 
the transformation $g$ generated by a strongly connected automaton is 
invertible, then $\mu$ and $g_*\mu$ are singular except for a few cases, in 
which $g_*\mu=\mu$.  We are going to correct his result fixing a minor 
error in the argument, and then further extend it.

One obvious exception is when $g$ acts trivially.  The second exception is 
when $\mu$ is the uniform Bernoulli measure (defined by a constant 
probability vector).  Such a measure is preserved by any invertible 
automaton transformation.  Unfortunately, another exceptional case (that 
kind of combines the said two) was overlooked in \cite{Krav}.

\begin{example}
Let $X=\{1,2,3\}$ and $\mu$ be a Bernoulli measure on $X^\nn$ defined by a 
probability vector $\bl=(1/2,1/4,1/4)$.  Let $\cA=(X,\{g\},\pi,\la)$, where 
$\pi(g,x) = g$ for all $x\in X$, $\lambda(g,1)=1$, $\lambda(g,2)=3$ and 
$\lambda(g,3)=2$.  The only state $g$ of the automaton $\cA$ acts on $X^\nn$
as a $1$-block factor map that applies the transposition $(2\,3)$ to every 
term of a sequence.  Since $\bl_2=\bl_3$, we have $g_*\mu=\mu$ even though 
$g$ does not act trivially and the measure $\mu$ is not uniform. \tri
\end{example}

\begin{lemma}\label{sing-exceptB}
Let $\mu$ be a Bernoulli measure on $X^\nn$ defined by a positive 
probability vector $\bl$ and $g:X^\nn\to X^\nn$ be an invertible 
transformation generated by a strongly connected automaton 
$\cA=(X,S,\pi,\la)$.  Then the following conditions are equivalent: (i) 
$g_*\mu=\mu$; (ii) $\bl_{x'}=\bl_x$ whenever $\la(s,x)=x'$ for some $s\in 
S$.
\end{lemma}

\begin{proof}
The action of $g$ on $X^*$ is invertible as well.  We use $g^{-1}$ to 
denote the inverse of both the action of $g$ on $X^\nn$ and on $X^*$.  
Consider an arbitrary word $u\in X^*$ of length $k\ge1$.  We have 
$u=x_1x_2\ldots x_k$ and $g^{-1}(u)=y_1y_2\ldots y_k$ for some $x_i,y_i\in 
X$, $1\le i\le k$.  Then $\mu(uX^\nn)=\bl_{x_1}\bl_{x_2}\dots\bl_{x_k}$ and 
$g_*\mu(uX^\nn)=\mu(g^{-1}(uX^\nn))=\mu\bigl(g^{-1}(u)X^\nn\bigr)=
\bl_{y_1}\bl_{y_2}\dots\bl_{y_k}$.  Note that $x_i=\la(s_i,y_i)$, where 
$s_1=g$ and $s_i=\pi(g,y_1y_2\ldots y_{i-1})$ for $2\le i\le k$.  Assuming 
the condition (ii) holds, we obtain that $\bl_{x_i}=\bl_{y_i}$ for $1\le 
i\le k$.  Then $\mu(uX^\nn)=g_*\mu(uX^\nn)$.  Thus the measures $\mu$ and 
$g_*\mu$ coincide on the cylinders, which implies that $g_*\mu=\mu$.

Conversely, assume that $g_*\mu=\mu$ and suppose $\la(s,x)=x'$ for some 
$s\in S$.  Since the automaton $\cA$ is strongly connected, there exists a 
word $u\in X^*$ such that $\pi(g,u)=s$.  Let $u'=g(u)$.  Then 
$g(ux)=u'x'$.  As a consequence, $\mu(uX^\nn)=g_*\mu(u'X^\nn)=\mu(u'X^\nn)$ 
and $\mu(uxX^\nn)=g_*\mu(u'x'X^\nn)=\mu(u'x'X^\nn)$.  By definition of the 
measure $\mu$, we have $\mu(uxX^\nn)=\mu(uX^\nn)\bl_x$ and 
$\mu(u'x'X^\nn)=\mu(u'X^\nn)\bl_{x'}$.  Since $\bl$ is a positive vector, 
the measure $\mu(uX^\nn)=\mu(u'X^\nn)$ is not zero.  It follows that 
$\bl_{x'}=\bl_x$.
\end{proof}

The main idea behind the proof of singularity is rather simple.  Recall 
that the asymptotic frequency $\freq_\om(u)$ with which a finite word $u\in 
X^*$ occurs in an infinite sequence $\om\in X^\nn$ is defined as a limit
\[
\freq_\om(u)= \lim_{n\to\infty} \frac1n \sum_{i=0}^{n-1} 
\chi_{uX^\nn}(\si^i(\om))
\]
(it is not defined if the limit does not exist).

\begin{lemma}\label{sing-2freq}
Let $\mu$ be a Borel probability measure on $X^\nn$ that is invariant and 
ergodic with respect to the shift.  Let $g:X^\nn\to X^\nn$ be a Borel 
measurable map.  Suppose that $\freq_{g(\om)}(u)\ne\freq_\om(u)$ for some 
$u\in X^*$ and $\mu$-almost all $\om\in X^\nn$.  Then the measures $\mu$ 
and $g_*\mu$ are singular.
\end{lemma}

\begin{proof}
Let $E_1$ be the set of all sequences $\om\in X^\nn$ such that 
$\freq_\om(u)=\mu(uX^\nn)$.  Let $E_2$ be the set of all $\om\in X^\nn$ 
such that $\freq_{g(\om)}(u)\ne\freq_\om(u)$.  Both $E_1$ and $E_2$ are 
Borel measurable sets.  We have $\mu(E_2)=1$ by assumption and $\mu(E_1)=1$ 
due to the Birkhoff ergodic theorem.  As a consequence, $\mu(E_1\cap E_2) 
=1$.  The image $g(E_1\cap E_2)$ is clearly disjoint from $E_1$.  It 
follows that $g_*\mu(X^\nn\setminus E_1)\ge\mu(E_1\cap E_2)=1$.  Hence 
$E_1$ is a set of full measure for $\mu$ while $X^\nn\setminus E_1$ is a 
set of full measure for $g_*\mu$.  Thus $\mu$ and $g_*\mu$ are singular 
measures.
\end{proof}

The next lemma is crucial for this section.

\begin{lemma}\label{sing-tensor}
Let $\mu$ be a Markov measure on $X^\nn$ defined by an irreducible 
stochastic matrix $L$ with stationary probability vector $\bl$ and 
$g:X^\nn\to X^\nn$ be an invertible transformation generated by an 
$L$-strongly connected automaton $\cA$.  Suppose that $\bt=\bk\otimes\bl$, 
where $\bt$ is the stationary probability vector of the matrix $T_{L,\cA}$ 
defined in \eqref{eq:T} and $\bk$ is the stationary probability vector of 
the matrix $K_{\bl,\cA}$ defined in \eqref{eq:K}.  Then the measures $\mu$ 
and $g_*\mu$ are either singular or the same.
\end{lemma}

\begin{proof}
Since $g$ is invertible, all restriction of $g$ are invertible as well.  
Since the automaton $\cA$ is strongly connected, every state $s\in S$ is a 
restriction of $g$.  Both the action of $s$ on $X^\nn$ and on $X^*$ are 
invertible.  We denote by $s^{-1}$ the inverses of both actions.

Assume $g_*\mu\ne\mu$.  Then $g_*\mu(wX^\nn)\ne\mu(wX^\nn)$ for some 
nonempty word $w\in X^*$.  Note that $g_*\mu(wX^\nn)=\mu(g^{-1}(wX^\nn)) 
=\mu\bigl(g^{-1}(w)X^\nn\bigr)$.  Let $k$ denote the length of $w$.  
Consider all words $u\in X^*$ of length $k$ such that 
$\mu\bigl(s^{-1}(u)X^\nn\bigr)\ne \mu(uX^\nn)$ for some $s\in S$ (one such 
word is $w$) and choose among them one with the largest value of 
$\mu(uX^\nn)$.  We claim that $\mu\bigl(s^{-1}(u)X^\nn\bigr)\le 
\mu(uX^\nn)$ for all $s\in S$ (by the choice of $u$, at least one of these 
inequalities is going to be strict).  Indeed, take any $s\in S$ and let 
$u^{(0)}=u$, $u^{(1)},u^{(2)},\dots$ be a sequence of words such that
$u^{(n+1)}=s^{-1}(u^{(n)})$ for all $n\ge0$.  Since the state $s$ acts as a 
permutation on the finite set of all words of length $k$, it follows that 
the sequence is periodic.  If $\mu(u^{(n)}X^\nn)>\mu(uX^\nn)$ for some $n$, 
then $\mu(u^{(n+1)}X^\nn)=\mu(u^{(n)}X^\nn)$ due to the choice of $u$.  
Therefore $\mu(u^{(1)}X^\nn)>\mu(uX^\nn)$ would imply $\mu(u^{(n)}X^\nn)= 
\mu(u^{(1)}X^\nn)>\mu(uX^\nn)$ for all $n\ge1$, which is not the case as 
$u$ occurs infinitely often in the sequence.

Let $u=u_1u_2\ldots u_k$, where each $u_i\in X$.  By Theorem 
\ref{mainthm1}, for $\mu$-almost all $\om\in X^\nn$ we have
\[
\freq_{g(\om)}(u)=\sum_{s_0\arr{x_0}{u_1} s_1\ldots\arr{x_{k-1}}{u_k} s_k}
\bt_{(s_0, x_0)} L_{x_0x_1}\ldots L_{x_{k-2}x_{k-1}},
\]
where the sum is over paths in the Moore diagram of $\cA$.  Let $\Sigma$ 
denote the value of the sum.  Since $\bt=\bk\otimes\bl$, we have
\[
\bt_{(s_0, x_0)} L_{x_0x_1}\ldots L_{x_{k-2}x_{k-1}}=
\bk_{s_0}\bl_{x_0} L_{x_0x_1}\ldots L_{x_{k-2}x_{k-1}}=
\bk_{s_0}\mu(x_0x_1\ldots x_{k-1}X^\nn).
\]
For any choice of $s_0$ the Moore diagram of $\cA$ admits a unique path of 
the form \,$s_0\arr{x_0}{u_1} s_1\ldots \arr{x_{k-1}}{u_k} s_k$, with 
$x_0x_1\ldots x_{k-1}=s_0^{-1}(u)$.  It follows that
\[
\Sigma=\sum_{s\in S}\bk_s\mu\bigl(s^{-1}(u)X^\nn\bigr).
\]
Since the automaton $\cA$ is strongly connected, the stochastic matrix 
$K_{\bl,\cA}$ is irreducible.  Therefore the vector $\bk$ is positive.  By 
the above, $\mu\bigl(s^{-1}(u)X^\nn\bigr)\le\mu(uX^\nn)$ for all $s\in S$.  
Moreover, at least one of these inequalities is strict.  It follows that 
$\Sigma<\sum_s \bk_s\mu(uX^\nn)=\mu(uX^\nn)$.  In particular, 
$\freq_{g(\om)}(u)<\mu(uX^\nn)$ for $\mu$-almost all $\om\in X^\nn$.  On 
the other hand, $\freq_\om(u)=\mu(uX^\nn)$ for $\mu$-almost all $\om\in 
X^\nn$ due to the Birkhoff ergodic theorem.  Now Lemma \ref{sing-2freq} 
implies that the measures $\mu$ and $g_*\mu$ are singular.
\end{proof}

\begin{theorem}\label{sing-mainB}
Let $\mu$ be a Bernoulli measure on $X^\nn$ defined by a positive 
probability vector $\bl$.  Suppose $g:X^\nn\to X^\nn$ is an invertible 
transformation generated by a strongly connected automaton 
$\cA=(X,S,\pi,\la)$.  Then the measures $\mu$ and $g_*\mu$ are singular 
unless $\bl_{\la(s,x)}=\bl_x$ for all $s\in S$ and $x\in X$, in which case 
$g_*\mu=\mu$.
\end{theorem}

\begin{proof}
The measure $\mu$ can be regarded as a Markov measure defined by a 
stochastic matrix $L$ each row of which coincides with $\bl$.  Since all 
entries of $L$ are positive, the automaton $\cA$ is $L$-strongly connected.
By Lemma \ref{markovirr}, the stochastic matrix $T_{L,\cA}$ defined in 
\eqref{eq:T} is irreducible.  Therefore its stationary probability vector 
$\bt$ is unique.  Lemma \ref{tensorB} implies that $\bt=\bk\otimes\bl$, 
where $\bk$ is the stationary probability vector of the stochastic matrix 
$K_{\bl,\cA}$ defined in \eqref{eq:K}.  By Lemma \ref{sing-tensor}, the 
measures $\mu$ and $g_*\mu$ are either singular or the same.  It follows 
from Lemma \ref{sing-exceptB} that $g_*\mu=\mu$ if and only if 
$\bl_{\la(s,x)}=\bl_x$ for all $s\in S$ and $x\in X$.
\end{proof}

\begin{example}
Let $X=\{1,2,3\}$ and $\mu$ be a Bernoulli measure on $X^\nn$ defined by a 
probability vector $\bl=(1/2,1/4,1/4)$.  Let $\cA=(X,\{s_0,s_1\},\pi,\la)$, 
where $\pi(s_i,x)=s_{1-i}$ for all $x\in X$ and $i\in\{0,1\}$, 
$\lambda(s_0,1)=2$, $\lambda(s_1,1)=3$, and $\lambda(s_i,x)=1$ for 
$x\in\{2,3\}$ and $i\in\{0,1\}$.  Let $g$ be either of the two states of 
the automaton $\cA$.  Then for $\mu$-almost all $\om\in X^\nn$ any symbol 
$x\in X$ occurs with the same frequency $\bl_x$ in $\om$ and $g(\om)$.  If 
$g$ were invertible, this would imply $g_*\mu=\mu$.  In fact, the measures 
$\mu$ and $g_*\mu$ are singular, but we need to look at words of length $2$ 
to be able to apply Lemma \ref{sing-2freq}.  Indeed, $22$ and $33$ occur 
with the same frequency $1/16$ in a $\mu$-generic sequence $\om$ while not 
occurring at all in $g(\om)$. \tri
\end{example}

In view of the previous example, we should expect the measures $\mu$ and 
$g_*\mu$ to be singular even if $g$ is not invertible.  There are 
exceptions, of course.

\begin{example}
Let $X$ be any alphabet of more than one character.  For any $x\in X$ and 
$\om\in X^\nn$ let $g_x(\om)=x\om$.  All transformations $g_x$, $x\in X$ 
can be generated by a single automaton $\cA=(X,S,\pi,\la)$, where 
$S=\{g_x\mid x\in X\}$, $\pi(g_x,y)=g_y$ and $\la(g_x,y)=x$ for all $x,y\in 
X$.  If $\mu$ is a Bernoulli measure on $X^\nn$ defined by a positive 
probability vector $\bl$, then $\mu=\sum_x \bl_x\,(g_{x})_*\mu$.  As a 
consequence, each measure $(g_{x})_*\mu$ is absolutely continuous with 
respect to $\mu$ while not the same as $\mu$. \tri
\end{example}

To prove an analogue of Theorem \ref{sing-mainB} for general Markov 
measures, we need first to establish an analogue of Lemma 
\ref{sing-exceptB}.

\begin{lemma}\label{sing-exceptM}
Let $\mu$ be a Markov measure on $X^\nn$ defined by an irreducible 
stochastic matrix $L$ with stationary probability vector $\bl$ and 
$g:X^\nn\to X^\nn$ be an invertible transformation generated by an 
$L$-strongly connected automaton $\cA=(X,S,\pi,\la)$.  Then the following 
conditions are equivalent: (i) $g_*\mu=\mu$; (ii) $\bl_{x'}=\bl_x$ whenever 
$\la(g,x)=x'$, and $L_{x'y'}=L_{xy}$ whenever $\la(s,xy)=x'y'$ for some 
$s\in S$.
\end{lemma}

\begin{proof}
Consider an arbitrary word $u=x_1x_2\ldots x_k\in X^*$ and let 
$g^{-1}(u)=y_1y_2\ldots y_k$.  Then $\mu(uX^\nn)= \bl_{x_1}L_{x_1x_2}\dots 
L_{x_{k-1}x_k}$ and $g_*\mu(uX^\nn)=\mu(g^{-1}(uX^\nn))= 
\mu\bigl(g^{-1}(u)X^\nn\bigr)=\bl_{y_1}L_{y_1y_2}\dots L_{y_{k-1}y_k}$.  
Clearly, $x_1=\la(g,y_1)$ and $x_1x_2=\la(g,y_1y_2)$.  Besides, 
$x_ix_{i+1}=\la(s_i,y_iy_{i+1})$ for $2\le i\le k-1$, where 
$s_i=\pi(g,y_1y_2\ldots y_{i-1})$ .  Assuming the condition (ii) holds, we 
obtain that $\bl_{x_1}=\bl_{y_1}$ and $L_{x_ix_{i+1}}=L_{y_iy_{i+1}}$ for 
$1\le i\le k-1$.  Then $\mu(uX^\nn)=g_*\mu(uX^\nn)$.  Thus the measures 
$\mu$ and $g_*\mu$ coincide on the cylinders, which implies that 
$g_*\mu=\mu$.

Conversely, assume that $g_*\mu=\mu$.  If $\la(g,x)=x'$ for some $x,x'\in 
X$, then $\bl_{x'}=\mu(x'X^\nn)=g_*\mu(x'X^\nn)=\mu(xX^\nn)=\bl_x$.  Now 
suppose $\la(s,xy)=x'y'$ for some $s\in S$.  Since the automaton $\cA$ is 
$L$-strongly connected, there exist symbols $x_0,x_1,\dots,x_k=x$ ($k\ge1$) 
such that $\pi(g,x_0x_1\ldots x_{k-1})=s$ and $L_{x_ix_{i+1}}>0$ for $0\le 
i\le k-1$.  Note that the vector $\bl$ is positive since the stochastic 
matrix $L$ is irreducible.  Therefore $\mu(x_0x_1\ldots x_kX^\nn)= 
\bl_{x_0}L_{x_0x_1}\dots L_{x_{k-1}x_k}>0$.  Let $u=x_0x_1\ldots x_{k-1}$ 
and $u'=g(u)$.  Since $\pi(g,u)=s$, we have $g(ux)=u'x'$ and 
$g(uxy)=u'x'y'$.  As a consequence, $\mu(uxX^\nn)=g_*\mu(u'x'X^\nn) 
=\mu(u'x'X^\nn)$ and $\mu(uxyX^\nn)=g_*\mu(u'x'y'X^\nn)=\mu(u'x'y'X^\nn)$.  
By definition of the measure $\mu$, we have 
$\mu(uxyX^\nn)=\mu(uxX^\nn)L_{xy}$ and 
$\mu(u'x'y'X^\nn)=\mu(u'x'X^\nn)L_{x'y'}$.  By the above the measure 
$\mu(uxX^\nn)=\mu(u'x'X^\nn)$ is not zero.  It follows that 
$L_{x'y'}=L_{xy}$.
\end{proof}

\begin{theorem}\label{sing-mainM}
Let $\mu$ be a Markov measure on $X^\nn$ defined by an irreducible 
stochastic matrix $L$ with stationary probability vector $\bl$.  Suppose 
$g:X^\nn\to X^\nn$ is an invertible transformation generated by a 
reversible, $L$-strongly connected automaton $\cA=(X,S,\pi,\la)$.  Then the 
measures $\mu$ and $g_*\mu$ are singular unless $\bl_{\la(g,x)}=\bl_x$ and 
$L_{\la(s,x),\,\la(\pi(s,x),y)}=L_{x,y}$ for all $s\in S$ and $x,y\in X$, 
in which case $g_*\mu=\mu$.
\end{theorem}

\begin{proof}
The theorem is proved in the same way as Theorem \ref{sing-mainB} but 
instead of Lemmas \ref{tensorB} and \ref{sing-exceptB}, one has to use 
respectively Lemmas \ref{tensor-rev} and \ref{sing-exceptM}.
\end{proof}

\newpage

\begin{raggedright}
\sc
Department of Mathematics\\
Mailstop 3368\\
Texas A\&M University\\
College Station, TX 77843-3368\\[3mm]
\emph{Email:\/} {\tt grigorch@math.tamu.edu}, {\tt 
romwell@math.tamu.edu},\\ 
{\tt yvorobet@math.tamu.edu}
\end{raggedright}

\end{document}